\documentclass{article}
\pdfoutput=1
\usepackage{euscript,amsmath,amssymb,amsfonts,graphicx,bm,amsthm}

  \usepackage{paralist}
  \usepackage{graphics} 
\usepackage[caption=false]{subfig}
\usepackage{soul}

\usepackage{latexsym,amssymb,enumerate,amsmath,amsthm} 
  \usepackage{graphicx} 
  \usepackage{tikz}
     \usepackage[colorlinks=false]{hyperref}
 \hypersetup{urlcolor=blue, citecolor=red}
\usepackage{latexsym,amssymb,enumerate,amsmath} 
\usepackage{pgfplots}
\usepackage{soul,xcolor}

\usepackage{mathtools}

\newcommand{\dd}{\text{d}}
\def\eps{\varepsilon}
\def\E{\mathbb{E}}
\def\P{\mathbb{P}}
\def\R{\mathbb{R}}
\def\tod{\to_{\textup{d}}}

\newcommand{\Markov}[2]{\underset{#1}{\overset{#2}{\rightleftharpoons}}}
\def\ts{t_{0}}
\def\tc{t_{\textup{c}}}
\def\xc{x_{\textup{c}}}
\def\taut{\widetilde{\tau}}
\def\Tt{\widetilde{T}}
\def\dist{\textup{d}}
\def\O{\mathcal{O}}
\def\oned{\textup{1d}}
\def\twod{\textup{2d}}
\def\threed{\textup{3d}}

\DeclarePairedDelimiter\floor{\lfloor}{\rfloor}

\newtheorem{theorem}{Theorem}
\newtheorem{proposition}[theorem]{Proposition}

\theoremstyle{plain}
\theoremstyle{remark}
\newtheorem{remark}[theorem]{Remark}
\theoremstyle{definition}
\newtheorem*{definition*}{Definition}

\begin{document}


\title{Extreme first passage times of piecewise deterministic Markov processes}


\author{Sean D. Lawley\thanks{Department of Mathematics, University of Utah, Salt Lake City, UT 84112 USA (\texttt{lawley@math.utah.edu}). The author was supported by the National Science Foundation (Grant Nos.\ DMS-1814832 and DMS-1148230).}
}
\date{\today}
\maketitle

\begin{abstract}
The time it takes the fastest searcher out of $N\gg1$ searchers to find a target determines the timescale of many physical, chemical, and biological processes. This time is called an extreme first passage time (FPT) and is typically much faster than the FPT of a single searcher. Extreme FPTs of diffusion have been studied for decades, but little is known for other types of stochastic processes. In this paper, we study the distribution of extreme FPTs of piecewise deterministic Markov processes (PDMPs). PDMPs are a broad class of stochastic processes that evolve deterministically between random events. Using classical extreme value theory, we prove general theorems which yield the distribution and moments of extreme FPTs in the limit of many searchers based on the short time distribution of the FPT of a single searcher. We then apply these theorems to some canonical PDMPs, including run and tumble searchers in one, two, and three space dimensions. We discuss our results in the context of some biological systems and show how our approach accounts for an unphysical property of diffusion which can be problematic for extreme statistics.
\end{abstract}

\section{Introduction}

The first time a random searcher finds a target is called a first passage time (FPT) and is commonly used to understand timescales in many areas of physics, chemistry, and biology \cite{redner2001}. The majority of prior work on FPTs analyzes the first time a given single searcher finds a target. However, it is being increasingly realized that in many applications the important timescale is not how long it takes a given single searcher to find a target, but rather how long it takes the fastest searcher out of $N\gg1$ searchers to find a target \cite{schuss2019, coombs2019, redner2019, sokolov2019, rusakov2019, martyushev2019, tamm2019, basnayake2019c}.

One particularly illustrative example is in human fertilization \cite{meerson2015}. Why do roughly $N=3\times10^{8}$ sperms cells search for the oocyte when only one sperm cell is necessary for fertilization? It is believed that a single sperm cell searching for the oocyte would be far too slow, and thus many sperm cells are required to accelerate the search process \cite{meerson2015, reynaud2015, schuss2019}. Remarkably, $N=3\times10^{8}$ sperm cells seems to be necessary, as a reduction in the number of sperm cells by only a factor of four may cause infertility \cite{bensdorp2007}. 

To setup the problem more precisely, let $\tau_{1},\dots,\tau_{N}$ be $N$ independent and identically distributed (iid) FPTs. Prior work has focused on a single FPT, $\tau_{1}$, whereas the more relevant timescale in many applications is the minimum FPT, 
\begin{align}\label{TN}
T_{N}
:=\min\{\tau_{1},\dots,\tau_{N}\}.
\end{align}
This minimum or fastest FPT, $T_{N}$, is called an \emph{extreme value} \cite{colesbook}, and it is typically much faster than $\tau_{1}$ if $N\gg1$.

The first motivation for the present work deals with a well-known \cite{keller2004} unphysical property of diffusion that is particularly problematic in extreme value theory. Most of the prior work on extreme FPTs deals with diffusive searchers, meaning $\tau_{1},\dots,\tau_{N}$ are the FPTs of $N$ independent diffusive Brownian searchers to find some target. If the searchers have diffusivity $D>0$, and the target is distance $L>0$ from the initial searcher locations, then the mean of this extreme FPT satisfies \cite{weiss1983, lawley2019uni}
\begin{align}\label{vanish}
\E[T_{N}]
\sim\frac{L^{2}}{4D\ln N}\quad\text{as }N\to\infty.
\end{align}
Notice that this time vanishes as $N\to\infty$. However, diffusion approximates a random walk with steps of finite speed. In particular, if a searcher follows a random walk with speed $v>0$, then it could never find a target that is distance $L$ away faster than time $\ts:=L/v>0$. Hence,
\begin{align}\label{lb}
T_{N}\ge \ts>0,
\end{align}
which contradicts \eqref{vanish}.

The discrepancy between \eqref{vanish} and \eqref{lb} stems from the infinite speed of propagation of solutions to the diffusion equation \cite{keller2004, kuske1997, joseph1989}. To illustrate, consider a one-dimensional (1d) searcher that moves at constant speed $v>0$ but switches direction (either to the left or to the right) at rate $\lambda>0$. This process is called a 1d run and tumble \cite{malakar2018, dhar2019}. If the searcher starts at the origin, then the probability density for its position, $p(x,t)$, satisfies
\begin{align}\label{zero}
p(x,t)=0,\quad |x|>vt,
\end{align}
since it cannot move more than distance $vt$ in time $t$. Now, it is common to approximate the density of a run and tumble process by solutions to the diffusion equation \cite{othmer2000},
\begin{align}\label{FPE}
\begin{split}
\frac{\partial}{\partial t}p
&=D\frac{\partial^{2}}{\partial x^{2}}p,
\end{split}
\end{align}
with $D=v^{2}/(2\lambda)$. However, solutions to \eqref{FPE} are strictly positive everywhere if $t>0$,
\begin{align*}
p(x,t)>0,\quad x\in\R,\,t>0,
\end{align*}
which violates \eqref{zero}. The problem is that \eqref{FPE} is a valid approximation for a 1d run and tumble process if $t\gg1/\lambda$ and $|x|=\O(\sqrt{(v^{2}/\lambda) t})$, whereas \eqref{zero} concerns values of $x$ outside the range of validity (see the Appendix for a review of this calculation).

The discrepancy between random walks at finite speed and diffusion at infinite speed can often be safely ignored in many applications, since the discrepancy occurs in the tails of the distribution. However, extreme FPTs depend precisely on these tails. In contrast to diffusion, a broad class of stochastic processes which move at finite speed and thus avoid this issue are piecewise deterministic Markov processes (PDMPs). 

PDMPs are stochastic processes that evolve deterministically between jumps of a Markov chain \cite{davis1984}. The word ``hybrid'' is often used in describing PDMPs since they consist of a continuous component $\{X(t)\}_{t\ge0}$ and a discrete component $\{J(t)\}_{t\ge0}$. The discrete component $J$ is a Markov jump process, and each element of its state space corresponds to some continuous dynamics for $X$. In between jumps of $J$, the continuous component $X$ evolves according to the dynamics associated with the current state of $J$. When $J$ jumps, $X$ switches to following the dynamics associated with the new state of $J$. Typically, $X$ takes values in $\R^{d}$ and follows an ordinary differential equation (ODE) in between jumps of $J$,
\begin{align}\label{gen}
\tfrac{\dd}{\dd t}X(t)=F_{J(t)}\big(X(t)\big)\in\R^{d},
\end{align}
where $\{F_{j}(x)\}_{j}$ is a given set of vector fields. Putting the 1d run and tumble described above in the framework of \eqref{gen}, $J(t)\in\{0,1\}$ jumps at rate $\lambda$, $d=1$, and the vector fields are simply $F_{0}(x)=-v$, $F_{1}(x)=v$.

In this paper, we study extreme FPTs for PDMPs. Specifically, we determine the distribution of $T_{N}$ in \eqref{TN} for large $N$ where the individual FPTs, $\tau_{1},\dots,\tau_{N}$, are iid realizations of the first time $X(t)$ in \eqref{gen} reaches some target $U_{\text{target}}\subset\R^{d}$,
\begin{align*}
\tau
:=\inf\{t>0:X(t)\in U_{\text{target}}\}.
\end{align*}
Since PDMPs move at finite speed (assuming $\sup_{j}\|F_{j}(x)\|<\infty$), we avoid the contradiction between \eqref{vanish} and \eqref{lb} which occurs for diffusion.

In addition, this work is motivated by the many applications of PDMPs in biology \cite{rudnicki17,cloez2017,bressloffbook,bressloff2017}, physics \cite{bena06, buceta02, horsthemke1984, doering1985, masoliver1986, doering1987}, engineering \cite{goebel04,hespanha14,teel15}, and finance \cite{yinbook}. In terms of biology, the swimming motion of bacteria is often modeled by a run and tumble process in two dimensions (2d) or three dimensions (3d) \cite{berg1993}. As in 1d, a run and tumble in 2d or 3d moves at constant speed (a ``run'') until a random ``tumbling'' time, at which point it chooses a new random direction and starts a new run until the next random tumbling time, and so on. Run and tumble processes in 1d have been used (i) in the Dogterom-Leibler model of microtubule catastrophes \cite{dogterom1993} and (ii) to model intracellular transport by molecular motors on a microtubule \cite{bressloff13}. Similar models have also been used to study sperm cells searching for an egg \cite{yang2016} and intermittent search strategies, in which a searcher switches between a slow search phase and a fast motile phase \cite{oshanin2009,benichou2011rev}. PDMPs are also used in stochastic gene expression \cite{smiley2010, lin2016}, biochemical reactions \cite{levien2017}, and neuroscience \cite{anderson2015}, where the continuous component $X$ represents the concentration(s) of abundant molecular species and the jump component $J$ represents either molecules with low copy number or some environmental state. PDMPs have also been used in ecology to understand population dynamics in a changing environment \cite{hening2018}. In addition, PDMPs have been studied for their interesting and sometimes counterintuitive mathematical features \cite{benaim12,hasler13,lawley14ode, bakhtin15, cloez15,lawley2018blowup, PB9, bakhtin2018}. Due to the diversity of the groups studying PDMPs, they are given several names in the literature, including stochastic hybrid systems, randomly switching dynamical systems, dichotomous Markov noise processes, velocity jump processes, and random evolutions.

The rest of the paper is organized as follows. In section~\ref{section main}, we summarize our main results. In section~\ref{math}, we give general theorems which yield the approximate distribution of extreme FPTs based on the short time asymptotic behavior of a single FPT. In sections~\ref{section 1d}-\ref{section 4B}, we apply these general results to four canonical PDMPs. We conclude by discussing related work and some biological applications. We collect the proofs and various technical details in the Appendix.

\begin{figure}[t]
\centering
\includegraphics[width=1\linewidth]{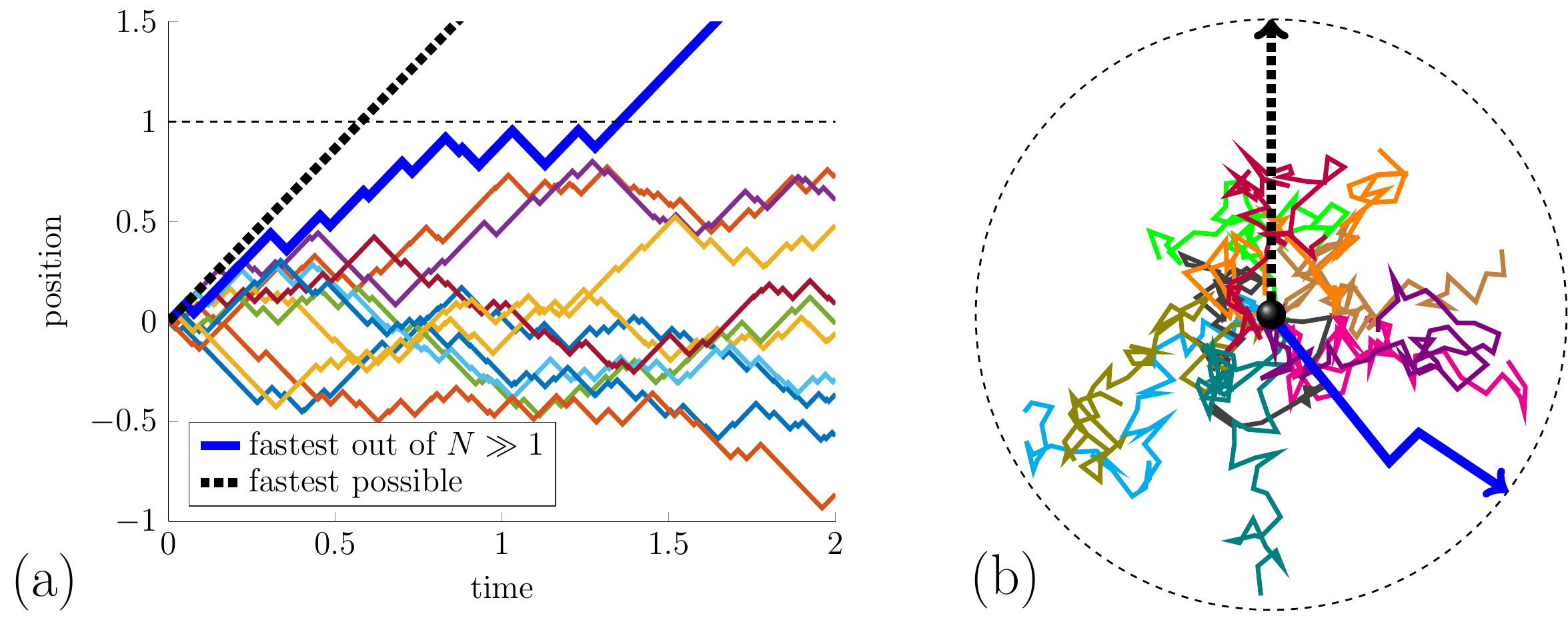}
\caption{(a) Run and tumble in 1d. (b) Run and tumble in 2d. In both (a) and (b), the black dashed curve illustrates the fastest possible trajectory, the thick blue curve illustrates the fastest trajectory out of $N\gg1$ trajectories, and the thin curves illustrate the $N-1$ slower trajectories.}
\label{figschem}
\end{figure}

\section{Main results}\label{section main}

Let $\{\tau_{n}\}_{n\ge1}$ be an iid sequence of realizations of a some FPT $\tau$, and assume that there exists $\ts>0$ and $q\in[0,1)$ so that
\begin{align}
\begin{split}\label{tq}
\P(\tau<\ts)
&=0,\\
\P(\tau=\ts)
&=q\in[0,1),\\
\P(\tau<\ts(1+\eps))
&>0\quad\text{for every }\eps>0.
\end{split}
\end{align}
The existence of such a $\ts$ is typical for FPTs of PDMPs, and it represents the fastest possible FPT. For example, in the case that $X$ is a run and tumble in 1d, 2d, or 3d with speed $v>0$, and $\tau$ is the first time the process escapes a ball of radius $L>0$,
\begin{align*}
\tau:=\inf\{t>0:\|X(t)\|>L\},
\end{align*}
then $\ts=L/v$. See Figure~\ref{figschem} for an illustration.

Define the fastest FPT $T_{N}$ as in \eqref{TN}. Since $\{\tau_{n}\}_{n\ge1}$ are iid, \eqref{tq} implies that
\begin{align}\label{qimp}
\begin{split}
\P(T_{N}<\ts)
&=0,\\
\P(T_{N}=\ts)
&=1-(1-q)^{N},\quad q\in[0,1).
\end{split}
\end{align}
Further, it follows from \eqref{tq} that $T_{N}$ converges almost surely to $\ts$ as $N\to\infty$ (even if $q=0$). To understand the distribution of $T_{N}$ for large $N$, we therefore need to understand the rate and distribution of the stochastic convergence of $T_{N}$ to $\ts$ as $N\to\infty$. To do this, we need information on the distribution of a single FPT $\tau$ around $\ts$. In particular, assume that there exists $\alpha>0$ and $p>0$ so that either
\begin{align}
\P(\ts<\tau<\ts(1+\eps))
&=(1-q)\alpha \eps^{p}+o(\eps^{p}),\label{short0}\\
\text{or}\quad\P(\ts<\tau<\ts(1+\eps))
&
=(1-q)\alpha\ln(1/\eps)\eps^{p}+o(\eps^{p}\ln(1/\eps)),\label{shortlog0}
\end{align}
as $\eps\to0+$, where $f(\eps)=o(g(\eps))$ means $\lim_{\eps\to0+}f/g=0$.

Under assumptions \eqref{tq} and \eqref{short0}-\eqref{shortlog0}, we prove below that (Theorem~\ref{main}) 
\begin{align}\label{coin}
T_{N}
=_{\dist}\ts(1+\xi_{N}a_{N}\Sigma_{N}),
\end{align}
where $=_{\dist}$ denotes equality in distribution, $\xi_{N}\in\{0,1\}$ is a Bernoulli random variable with
\begin{align}\label{coin2}
\P(\xi_{N}=1)=(1-q)^{N},
\end{align}
the scaling constant $a_{N}>0$ is
\begin{align}
a_{N}
:=\begin{cases}
(\alpha N)^{-1/p} & \text{if \eqref{short0} holds},\\
(\alpha N\ln(N)/p)^{-1/p} & \text{if \eqref{shortlog0} holds},
\end{cases}
\end{align}
and $\Sigma_{N}>0$ is a random variable independent of $\xi_{N}$ that converges in distribution to a Weibull random variable with unit scale and shape $p$,
\begin{align}\label{coin3}
\Sigma_{N}
\to_{\dist}
\textup{Weibull}(1,p)\quad\text{as }N\to\infty.
\end{align}

In words, \eqref{coin} means that the distribution of $T_{N}$ is given by (i) flip a coin to see if $T_{N}=\ts$ or $T_{N}>\ts$ (corresponding to $\xi_{N}=0$ or $\xi_{N}=1$) and (ii) if $T_{N}>\ts$, then $T_{N}=\ts(1+a_{N}\Sigma_{N})$ where $\Sigma_{N}$ is approximately a Weibull random variable for $N\gg1$. Put another way,
\begin{align*}
T_{N}
=_{\dist}
\begin{cases}
\ts & \text{with probability }1-(1-q)^{N}\\
\ts(1+a_{N}\Sigma_{N}) & \text{with probability }(1-q)^{N},
\end{cases}
\end{align*}
where $\Sigma_{N}>0$ is approximately Weibull for large $N$.

Therefore, \eqref{coin}-\eqref{coin3} give the full distribution of $T_{N}$ for large $N$. For example, \eqref{coin}-\eqref{coin3} yield all the moments of $T_{N}$ for large $N$. To illustrate, we prove that (Theorem~\ref{moments})
\begin{align*}
\E[T_{N}]
&= \ts+(1-q)^{N}\ts a_{N}\Gamma(1+1/p)+\textup{h.o.t.},\\
\textup{Variance}(T_{N})
&=({\ts a_{N}})^{2}(1-q)^{N}\Big[\Gamma(1+2/p)-(1-q)^{N}\big(\Gamma(1+1/p)\big)^{2}\Big]+\textup{h.o.t.},
\end{align*}
where $\Gamma(\cdot)$ is the Gamma function and h.o.t.\ refers to terms which are higher order for large $N$. We prove analogous results (Theorems~\ref{kth} and \ref{kth moment}) about the full distribution and moments of the $k$th fastest FPT,
\begin{align*}
T_{k,N}
:=\min\big\{\{\tau_{1},\dots,\tau_{N}\}\backslash\cup_{j=1}^{k-1}\{T_{j,N}\}\big\},\quad k\in\{1,\dots,N\},
\end{align*}
where $T_{1,N}:=T_{N}$.

We apply these general theorems to some specific PDMPs in sections~\ref{section 1d}-\ref{section 4B}. To illustrate briefly, consider run and tumble processes in 1d, 2d, and 3d with speed $v>0$ and tumbling rate $\lambda>0$ (for simplicity, upon tumbling in 2d or 3d, assume the new angle is chosen uniformly). Let $T_{N}^{\oned}$, $T_{N}^{\twod}$, and $T_{N}^{\threed}$ denote the first time one of $N$ such run and tumble processes escapes a ball of radius $L>0$ in 1d, 2d, and 3d. Defining the dimensionless tumbling rate $\rho:=\lambda L/v$, we find that 
\begin{align*}
\E[T_{N}^{\oned}]
&=\frac{L}{v}\left[1+\Big(\frac{2(1-e^{-{\rho}})^{N+1}}{{\rho}({\rho} +1) e^{-{\rho}}}\Big)\frac{1}{N}\right]
+\textup{h.o.t.},\\
\E[T_{N}^{\twod}]
&=\frac{L}{v}\left[1+\Big(\frac{(1-e^{-{\rho}})^{N+2}}{({\rho} e^{-{\rho}})^{2}}\Big)\frac{1}{N^{2}}\right]
+\textup{h.o.t.},\\
\E[T_{N}^{\threed}]
&=\frac{L}{v}\left[1+\Big(\frac{(1-e^{-{\rho}})^{N+1}}{{\rho} e^{-{\rho}} }\Big)\frac{1}{N\ln N}\right]
+\textup{h.o.t.}
\end{align*}
In addition to the means of $T_{N}^{\oned}$, $T_{N}^{\twod}$, and $T_{N}^{\threed}$, we also find their approximate full distributions for large $N$. We also apply our theorems to a PDMP whose velocity varies between switches (see section~\ref{section 4B}). The specific PDMP in section~\ref{section 4B} has been used to study both gene expression \cite{smiley2010} and storage systems \cite{boxma2005}.


\section{General theory}\label{math}

\subsection{Fastest FPT}

As in section~\ref{section main}, assume \eqref{tq} and \eqref{short0}-\eqref{shortlog0}. In light of \eqref{qimp}, in order to understand the distribution of $T_{N}$, it remains to understand the distribution of $T_{N}$ when $T_{N}>\ts$. Toward this end, let $\{\taut_{n}\}_{n\ge1}$ be an iid sequence of realizations of $\tau$ conditioned that $\tau>\ts$. That is,
\begin{align}\label{ccdf}
\P(\taut_{n}<t)
=\P(\tau<t|\tau>\ts)
=\frac{\P(\ts<\tau<t)}{\P(\tau>\ts)}
=\frac{\P(\ts<\tau<t)}{1-q}.
\end{align}
Further, define
\begin{align*}
\Tt_{N}
:=\min\{\taut_{1},\dots,\taut_{N}\}.
\end{align*}
Hence, we have
\begin{align}\label{tilde}
T_{N}
&=_{\dist}\ts+\xi_{N}(\Tt_{N}-\ts),
\end{align}
where $\xi_{N}\in\{0,1\}$ is an independent Bernoulli random variable satisfying
\begin{align*}
\P(\xi_{N}=1)=(1-q)^{N}.
\end{align*}

Therefore, it remains to understand the distribution of $\Tt_{N}-\ts$. We prove below that a certain rescaling of $\Tt_{N}-\ts$ converges in distribution to a Weibull random variable. The definition and proposition below give some standard facts about the Weibull distribution.

\begin{definition*}
A random variable $X\ge0$ has a \emph{Weibull distribution} with scale parameter $t>0$ and shape parameter $p>0$ if
\begin{align}\label{xweibull}
\P(X>x)
=\exp(-(x/t)^{p}),\quad x\ge0.
\end{align}
If \eqref{xweibull} holds, then we write
\begin{align*}
X=_{\textup{d}}\textup{Weibull}(t,p).
\end{align*}
If \eqref{xweibull} holds with $p=1$, then $X\ge0$ has an \emph{exponential distribution} with mean $t>0$, and we write
\begin{align*}
X=_{\textup{d}}\textup{Exponential}(t).
\end{align*}
\end{definition*}

\begin{proposition}\label{basic}
If $X=_{\textup{d}}\textup{Weibull}(t,p)$, then its survival probability is in \eqref{xweibull} and its probability density function is
\begin{align*}
f_{X}(x)
=(p/t)(x/t)^{p-1}\exp(-(x/t)^{p}),\quad x>0,
\end{align*}
with $f_{X}(x)=0$ if $x\le0$. Its moments are
\begin{align*}
\E[X^{m}]
=t^{m}\Gamma(1+m/p),\quad m\ge0,
\end{align*}
where $\Gamma(\cdot)$ denotes the gamma function. Hence, its mean and variance are
\begin{align*}
\E[X]
=t\Gamma(1+1/p),
\quad
\textup{Variance}(X)
=t^{2}\Big[\Gamma(1+2/p)-\big(\Gamma(1+1/p)\big)^{2}\Big].
\end{align*}
\end{proposition}

The first theorem below proves that the following rescaling of $\Tt_{N}$ converges in distribution to a Weibull random variable with unit scale and shape $p$,
\begin{align*}
\Sigma_{N}
:=\frac{\Tt_{N}-t_{0}}{a_{N}t_{0}}
\to_{\dist}
\textup{Weibull}(1,p)\quad\text{as }N\to\infty,
\end{align*}
where $a_{N}$ and $p$ depend on the asymptotic distribution in \eqref{short0}-\eqref{shortlog0} of a single unconditioned FPT. In light of \eqref{tilde}, this gives the full distribution of $T_{N}$ for large $N$. The proofs of all the results of this section are collected in the Appendix.

\begin{theorem}\label{main}
Let $\{\tau_{n}\}_{n\ge1}$ be an iid sequence of random variables and assume
\begin{align*}
\P(\tau_{n}<\ts)=0,\quad \P(\tau_{n}=\ts)=q,\quad\text{for some $\ts>0$ and $q\in[0,1)$},
\end{align*}
and assume that for some $\alpha>0$ and $p>0$, either
\begin{align}
\P(\ts<\tau<\ts(1+\eps))
&=(1-q)\alpha \eps^{p}+o(\eps^{p}),\label{short}\\
\text{or}\quad\P(\ts<\tau<\ts(1+\eps))
&=(1-q)\alpha\ln(1/\eps)\eps^{p}+o(\eps^{p}\ln(1/\eps)),\label{shortlog}
\end{align}
as $\eps\to0+$, where $f(\eps)=o(g(\eps))$ means $\lim_{\eps\to0+}f/g=0$.

Then
\begin{align}\label{maineq}
T_{N}
:=\min\{\tau_{1},\dots,\tau_{N}\}
&=_{\dist}\ts(1+ a_{N}\xi_{N}\Sigma_{N}),
\end{align}
where
\begin{align}\label{aN}
a_{N}
:=\begin{cases}
(\alpha N)^{-1/p} & \text{if \eqref{short} holds},\\
(\alpha N\ln(N)/p)^{-1/p} & \text{if \eqref{shortlog} holds},
\end{cases}
\end{align}
and $\xi_{N}\in\{0,1\}$ is a Bernoulli random variable satisfying
\begin{align*}
\P(\xi_{N}=1)=(1-q)^{N},
\end{align*}
and $\Sigma_{N}>0$ is a random variable that converges in distribution to a Weibull random variable with unit scale and shape $p$,
\begin{align}\label{cw}
\Sigma_{N}
\to_{\dist}
\textup{Weibull}(1,p)\quad\text{as }N\to\infty.
\end{align}
The random variables $\xi_{N}$ and $\Sigma_{N}$ are independent.
\end{theorem}

\begin{remark}
If \eqref{maineq} and \eqref{cw} hold for some sequence of scalings $\{a_{N}\}$, then it is straightforward to check that \eqref{maineq} and \eqref{cw} hold for any sequence $\{a_{N}'\}$ satisfying
\begin{align*}
\lim_{N\to\infty}a_{N}/a_{N}'=1.
\end{align*}
Therefore, there are infinitely many valid choices of the scalings $\{a_{N}\}$ in Theorem~\ref{main}. The choices in \eqref{aN} are merely the simplest.

We further note that the values of $\{a_{N}\}$ are determined by inverting the function $S(\eps):=\P(\ts<\tau<\ts(1+\eps))$ for $\eps\ll1$ (see the proof of Theorem~\ref{main} in the Appendix for details). In the case that \eqref{shortlog} holds, the value of $a_{N}$ that results from this inversion is
\begin{align}\label{lambertw}
a_{N}
=
\Big(-\alpha N W_{-1}\big(-p/(N\alpha)\big)/p\Big)^{-1/p},
\end{align}
where $W_{-1}(z)$ denotes the lower branch of the LambertW function \cite{corless1996}. The value $a_{N}=(\alpha N\ln(N)/p)^{-1/p}$ in \eqref{aN} results from finding the asymptotic behavior of \eqref{lambertw} for large $N$.
\end{remark}

While Theorem~\ref{main} regards convergence in distribution, it is known that convergence in distribution does not necessarily imply moment convergence \cite{billingsley2013}. That is, for a sequence of random variables $\{X_{N}\}_{N\ge1}$, we cannot in general conclude that $\E[(X_{N})^{m}]\to\E[X^{m}]$ as $N\to\infty$ merely because $X_{N}\tod X$ as $N\to\infty$. However, in the case of extreme values, convergence in distribution does imply moment convergence \cite{pickands1968}.

\begin{theorem}\label{moments}
Under the assumptions of Theorem~\ref{main}, assume further that
\begin{align*}
\E[T_{N}]<\infty\quad\text{for some }N\ge1.
\end{align*}
Then for each moment $m\ge0$, we have that
\begin{align*}
\E[(\Sigma_{N})^{m}]
\to\Gamma(1+m/p)\quad\text{as }N\to\infty.
\end{align*}
Therefore, if $m\ge0$, then
\begin{align*}
\E[(T_{N}-\ts)^{m}]
&= (1-q)^{N}({\ts a_{N}})^{m}\Gamma(1+m/p)\\
&\quad+o\Big((1-q)^{N}({\ts a_{N}})^{m}\Big)\quad\text{as }N\to\infty,
\end{align*}
where $f(N)=o(g(N))$ means $\lim_{N\to\infty}f/g=0$. Hence,
\begin{align*}
\E[T_{N}]
&= \ts+(1-q)^{N}\ts a_{N}\Gamma(1+1/p)\\
&\quad+o\big((1-q)^{N}\ts a_{N}\big)\quad\text{as }N\to\infty,\\
\textup{Variance}(T_{N})
&=({\ts a_{N}})^{2}(1-q)^{N}\Big[\Gamma(1+2/p)-(1-q)^{N}\big(\Gamma(1+1/p)\big)^{2}\Big]\\
&\quad+o\big((1-q)^{N}({\ts a_{N}})^{2}\big)\quad\text{as }N\to\infty.
\end{align*}
\end{theorem}

\subsection{$k$th fastest FPT}

In this subsection, we generalize Theorems~\ref{main} and \ref{moments} on the fastest FPT to the $k$th fastest FPT,
\begin{align*}
T_{k,N}
:=\min\big\{\{\tau_{1},\dots,\tau_{N}\}\backslash\cup_{j=1}^{k-1}\{T_{j,N}\}\big\},\quad k\in\{1,\dots,N\},
\end{align*}
where $T_{1,N}:=T_{N}$. The distribution of $T_{k,N}$ can be described in terms of a generalized Gamma distribution.

\begin{definition*}
A random variable $X\ge0$ has a \emph{generalized Gamma distribution} with parameters $t>0$, $p>0$, $k>0$ if
\begin{align}\label{gGamma}
\P(X>x)
=\frac{\Gamma(k,(x/t)^{p})}{\Gamma(k)},\quad x\ge0,
\end{align}
where $\Gamma(a,z):=\int_{z}^{\infty}u^{a-1}e^{-u}\,\dd u$ denotes the upper incomplete gamma function. If \eqref{gGamma} holds, then we write
\begin{align*}
X=_{\textup{d}}\textup{gen}\Gamma(t,p,k).
\end{align*}
If \eqref{gGamma} holds with $p=1$ and $k\in\{1,2,3,\dots\}$, then $X\ge0$ has an \emph{Erlang distribution} with scale $t>0$ and shape $k$, and we write
\begin{align*}
X=_{\textup{d}}\textup{Erlang}(t,k).
\end{align*}
\end{definition*}

\begin{proposition}\label{basicg}
If $X=_{\textup{d}}\textup{gen}\Gamma(t,p,k)$, then its survival probability is in \eqref{gGamma} and its probability density function is
\begin{align*}
f_{X}(x)
=\frac{p(x/t)^{kp}\exp(-(x/t)^{p})}{x\Gamma(k)},\quad x>0,
\end{align*}
with $f_{X}(x)=0$ if $x\le0$. Its moments are
\begin{align*}
\E[X^{m}]
=t^{m}\frac{\Gamma(k+m/p)}{\Gamma(k)},\quad m\ge0.
\end{align*}
Hence, its mean and variance are
\begin{align*}
\E[X]
=t\frac{\Gamma(k+1/p)}{\Gamma(k)},
\quad
\textup{Variance}(X)
=t^{2}\left[\frac{\Gamma(k+2/p)}{\Gamma(k)}-\Big(\frac{\Gamma(k+1/p)}{\Gamma(k)}\Big)^{2}\right].
\end{align*}
Further, if $X=_{\textup{d}}\textup{Erlang}(t,k)$, then $X$ is equal in distribution to a sum of $k$ iid exponential random variables,
\begin{align*}
X
=_{\textup{d}}\sum_{j=1}^{k}X_{j},\quad\text{where }X_{j}=_{\dist}\textup{Exponential}(t).
\end{align*}
\end{proposition}

The following theorem gives the distribution of $T_{k,N}$ for large $N$.

\begin{theorem}\label{kth}
Under the assumptions of Theorem~\ref{main}, we have that
\begin{align}\label{tknsum}
T_{k,N}
=_{\dist}\ts\Big(1+{a_{N}}\sum_{j=0}^{k-1}\xi_{j,N}\Sigma_{k-j,N}\Big),
\end{align}
where the random variables $\{\xi_{j,N}\}_{j=0}^{N}$ satisfy $\xi_{j,N}\in\{0,1\}$, $\sum_{j=0}^{N}\xi_{j,N}=1$, and 
\begin{align*}
\P(\xi_{j,N}=1)
={N\choose j}q^{j}(1-q)^{N-j},
\end{align*}
and $\{\Sigma_{k-j,N}\}_{j=0}^{k-1}$ are nonnegative random variables that converge in distribution to generalized Gamma random variables,
\begin{align}\label{cd1}
\Sigma_{k-j,N}
\to_{\dist}
\textup{gen}\Gamma(1,p,k-j)\quad\text{as }N\to\infty.
\end{align}
The random variables $\{\xi_{j,N}\}_{j=0}^{N}$ are independent of $\{\Sigma_{k-j,N}\}_{j=0}^{k-1}$.
\end{theorem}

The random variables $\{\xi_{j,N}\}_{j=0}^{k-1}$ in Theorem~\ref{kth} serve to indicate how many of the $k$ fastest FPTs, $T_{1,N},\dots,T_{k,N}$, are equal to $\ts$. In particular, the event $\xi_{j,N}=1$ (and thus $\xi_{i,N}=0$ for $i\neq j$) corresponds to the event that exactly $j\in\{0,\dots,k-1\}$ of the $k$ fastest FPTs are equal to $\ts$,
\begin{align*}
T_{1,N}=T_{2,N}=\dots=T_{j,N}=\ts,\quad T_{j+1,N}>\ts.
\end{align*}
Hence, if $\xi_{j,N}=1$, then \eqref{tknsum} becomes
\begin{align*}
T_{k,N}
=_{\dist}\ts(1+a_{N} \Sigma_{k-j,N})
\end{align*}
and $\Sigma_{k-j,N}$ is a rescaling of $\Tt_{k-j,N-j}$
\begin{align*}
\Sigma_{k-j,N}
=\frac{\Tt_{k-j,N-j}-\ts}{a_{N}\ts},
\end{align*}
where we define $\Tt_{k,N}$ as the $k$th fastest FPT out of the $N$ conditioned FPTs,
\begin{align*}
\Tt_{k,N}
:=\min\big\{\{\taut_{1},\dots,\taut_{N}\}\backslash\cup_{j=1}^{k-1}\{\Tt_{j,N}\}\big\},\quad k\in\{1,\dots,N\},
\end{align*}
where $\Tt_{1,N}:=\Tt_{N}$. Since Theorem~\ref{kth} concerns the limit $N\to\infty$, replacing $\Tt_{k-j,N-j}$ by $\Tt_{k-j,N}$ is immaterial. Also, if $\xi_{j,N}=1$ for $j\ge k$, then \eqref{tknsum} merely reduces to
\begin{align*}
T_{k,N}=\ts,
\end{align*}
since this corresponds to the event that at least $k$ many FPTs were equal to $\ts$.

The following theorem ensures the convergence of the moments of the $k$th fastest FPT.

\begin{theorem}\label{kth moment}
Under the assumptions of Theorem~\ref{moments}, we have that for each moment $m\ge0$,
\begin{align*}
\E[(\Sigma_{k-j,N})^{m}]
\to\frac{\Gamma(k-j+m/p)}{\Gamma(k-j)}\quad\text{as }N\to\infty.
\end{align*}
Therefore, if $m\ge0$ and $q\in(0,1)$, then
\begin{align*}
\E[(T_{k,N}-\ts)^{m}]
&= ({\ts a_{N}})^{m}\sum_{j=0}^{k-1}{N\choose j}q^{j}(1-q)^{N-j}\frac{\Gamma(k-j+m/p)}{\Gamma(k-j)}\\
&\quad+o\Big(({\ts a_{N}})^{m}\frac{(Nq)^{k-1}(1-q)^{N-k+1}}{(k-1)!}\Big)\quad\text{as }N\to\infty.
\end{align*}
Similarly, if $m\ge0$ and $q=0$, then
\begin{align*}
\E[(T_{k,N}-\ts)^{m}]
&= ({\ts a_{N}})^{m}\frac{\Gamma(k+m/p)}{\Gamma(k)}+o\big(({\ts a_{N}})^{m}\big)\quad\text{as }N\to\infty.
\end{align*}
\end{theorem}

\section{One-dimensional run and tumble}\label{section 1d}

In the following four sections, we apply the general theory of section~\ref{math} to extreme FPTs of some specific PDMPs. In order to apply the theory to each example, we merely need to show that the distribution of a single FPT satisfies either \eqref{short} or \eqref{shortlog} for some $\ts>0$, $q\in[0,1)$, $\alpha>0$, and $p>0$.

Consider a 1d run and tumble process (commonly called a velocity jump process) that moves with a constant velocity to the left or a constant velocity to the right and switches direction at a constant Poissonian rate. Specifically, consider the PDMP $(X(t),J(t))\in\R\times\{0,1\}$ satisfying
\begin{align}\label{1dvjp}
\begin{split}
\frac{\dd}{\dd t}X(t)
&=\begin{cases}
-v_{0}<0 & \text{if }J(t)=0,\\
v_{1}>0 & \text{if }J(t)=1,
\end{cases}\\
X(0)
&=0,
\end{split}
\end{align}
where $J(t)\in\{0,1\}$ is a two-state continuous-time Markov jump process with jump rates
\begin{align*}
0\Markov{\lambda_{1}}{\lambda_{0}}1.
\end{align*}
Let $p_{0}$ and $p_{1}$ give the initial distribution of $J$,
\begin{align*}
\P(J(0)=0)=p_{0}=1-p_{1}\in[0,1].
\end{align*}

\subsection{FPT to $L>0$}
Define the FPT of the process \eqref{1dvjp} to $L>0$,
\begin{align*}
\tau
:=\inf\{t>0:X(t)=L\}.
\end{align*}
Since the process moves with velocity either $-v_{0}<0$ or $v_{1}>0$, the smallest $\tau$ could be is
\begin{align*}
\ts
:=L/v_{1}>0.
\end{align*}
Further, 
\begin{align*}
q:=\P(\tau=\ts)
=p_{1}e^{-\lambda_{1}\ts}\in[0,1).
\end{align*}
In words, $q$ is the probability that the process starts in the positive direction ($p_{1}\in[0,1]$) and that the process does not change direction before hitting $L$ ($e^{-\lambda_{1}t_{0}}\in(0,1)$).

Let ${M}(t)$ be the number of jumps of $J$ before time $t$ and let $\{s_{n}\}_{n\ge1}$ be the sequence of holding times of $J$ (and thus $J$ jumps at times $s_{1},s_{1}+s_{2},s_{1}+s_{2}+s_{3},\dots$). Then
\begin{align}\label{sum}
\P(\ts<\tau<\ts(1+\eps))
=\sum_{j=0}^{\infty}\P(\ts<\tau<\ts(1+\eps),{M}(\ts)=j).
\end{align}
It is immediate that the $j=0$ term is zero for $\eps$ sufficiently small (namely $\eps<1$),
\begin{align}\label{j0}
\P(\ts<\tau<\ts(1+\eps),{M}(\ts)=0)=0,\quad\text{if }0<\eps<1,
\end{align}
because if ${M}(\ts)=0$ and $J(0)=1$, then $\tau=\ts$ and if ${M}(\ts)=0$ and $J(0)=0$, then $\tau\ge2\ts$ almost surely. Further, it is straightforward to check that if $j\ge3$, then
\begin{align}\label{j3}
\P(\ts<\tau<\ts(1+\eps),{M}(\ts)=j)=\O(\eps^{2}),\quad\text{if }j\ge3.
\end{align}

We thus focus on the $j=1,2$ terms in the sum in \eqref{sum}. Looking first at the $j=1$ term, it is straightforward to check that
\begin{align*}
\P(\ts<\tau<\ts(1+\eps),{M}(\ts)=1,J(0)=1)=\O(\eps^{2}).
\end{align*}
It is also immediate that
\begin{align*}
&\P(\ts<\tau<\ts(1+\eps),{M}(\ts)=1|J(0)=0)\\
&\quad=\P(\ts<\tau<\ts(1+\eps),{M}(\ts(1+\eps))=1|J(0)=0)+\O(\eps^{2}).
\end{align*}
Now, if $\ts<\tau<\ts(1+\eps)$, ${M}(\ts(1+\eps))=1$, and $J(0)=0$, then we need that
\begin{align*}
-s_{1}v_{0}+(\ts(1+\eps)-s_{1})v_{1}>L,
\end{align*}
which is equivalent to $s_{1}<\frac{\eps L}{v_{0}+v_{1}}$ (using that $\ts=L/v_{1}$). Hence,
\begin{align*}
&\P(\ts<\tau<\ts(1+\eps),{M}(\ts(1+\eps))=1|J(0)=0)\\
&\quad=\P(s_{1}<\tfrac{\eps L}{v_{0}+v_{1}},s_{2}>\ts(1+\eps)-s_{1}|J(0)=0)\\
&\quad=\P(s_{1}<\tfrac{\eps L}{v_{0}+v_{1}},s_{2}>\ts|J(0)=0)+\O(\eps^{2})\\
&\quad=\big[1-e^{-\lambda_{0}(\frac{\eps L}{v_{0}+v_{1}})}\big]e^{-\lambda_{1}\ts}+\O(\eps^{2}).
\end{align*}
Putting this together, we have that
\begin{align}\label{j1}
\P(\ts<\tau<\ts(1+\eps),{M}(\ts)=1)
=p_{0}\big[1-e^{-\lambda_{0}(\frac{\eps L}{v_{0}+v_{1}})}\big]e^{-\lambda_{1}\ts}+\O(\eps^{2}).
\end{align}

Moving to the $j=2$ term in \eqref{sum}, we have
\begin{align*}
\P(\ts<\tau<\ts(1+\eps),{M}(\ts)=2,J(0)=0)=\O(\eps^{2}),
\end{align*}
and
\begin{align*}
&\P(\ts<\tau<\ts(1+\eps),{M}(\ts)=2|J(0)=1)\\
&\quad=\P(\ts<\tau<\ts(1+\eps),{M}(\ts(1+\eps))=2|J(0)=1)+\O(\eps^{2}).
\end{align*}
Now, if $\ts<\tau<\ts(1+\eps)$, ${M}(\ts(1+\eps))=2$, and $J(0)=1$, then we need that
\begin{align*}
s_{1}v_{1}-v_{0}s_{2}+(\ts(1+\eps)-(s_{1}+s_{2}))v_{1}>L,
\end{align*}
which is equivalent to $s_{2}<\frac{\eps L}{v_{0}+v_{1}}$ (again using that $\ts=L/v_{1}$). Hence,
\begin{align*}
&\P(\ts<\tau<\ts(1+\eps),{M}(\ts(1+\eps))=2|J(0)=1)\\
&\quad=\P(s_{2}<\tfrac{\eps L}{v_{0}+v_{1}},s_{1}+s_{2}<\ts(1+\eps),s_{1}+s_{2}+s_{3}>\ts(1+\eps)|J(0)=1)\\
&\quad=\P(s_{2}<\tfrac{\eps L}{v_{0}+v_{1}},s_{1}<\ts,s_{1}+s_{3}>\ts|J(0)=1)
+\O(\eps^{2})\\
&\quad=\P(s_{2}<\tfrac{\eps L}{v_{0}+v_{1}}|J(0)=1)\P(s_{1}<\ts,s_{1}+s_{3}>\ts|J(0)=1)
+\O(\eps^{2})\\
&\quad=\big[1-e^{-\lambda_{0}(\frac{\eps L}{v_{0}+v_{1}})}\big]\P(s_{1}<\ts,s_{1}+s_{3}>\ts|J(0)=1)
+\O(\eps^{2}).
\end{align*}
Now, 
\begin{align*}
\P(s_{1}<\ts,s_{1}+s_{3}>\ts|J(0)=1)
=\int_{0}^{\ts}e^{-\lambda_{1}(\ts-s)}\lambda_{1}e^{-\lambda_{1}s}\,\dd s
=\lambda_{1}\ts e^{-\lambda_{1}\ts}.
\end{align*}
Putting this together, we obtain
\begin{align}\label{j2}
\P(\ts<\tau<\ts(1+\eps),{M}(\ts)=2)
=p_{1}\lambda_{1}\ts e^{-\lambda_{1}\ts}\big[1-e^{-\lambda_{0}(\frac{\eps L}{v_{0}+v_{1}})}\big]+\O(\eps^{2}).
\end{align}

Therefore, by \eqref{sum}, \eqref{j0}, \eqref{j3}, \eqref{j1}, and \eqref{j2}, we have that
\begin{align*}
\P(\ts<\tau<\ts(1+\eps))
&=(p_{0}+p_{1}\lambda_{1}\ts)e^{-\lambda_{1}\ts}\Big[1-e^{-\lambda_{0}(\frac{\eps L}{v_{0}+v_{1}})}\Big]
+\O(\eps^{2})\\
&=(1-q)\alpha\eps+\O(\eps^{2}),
\end{align*}
where
\begin{align}\label{alpha1}
\alpha
=\frac{\lambda_{0} L (\lambda_{1} L p_{1}-p_{1} v_{1}+v_{1})}{v_{1} (v_{0}+v_{1}) (e^{\lambda_{1} L/v_{1}}-p_{1})}.
\end{align}
Summarizing, \eqref{short} holds with $p=1$, $t_{0}=L/v_{1}$, $q=p_{1}e^{-\lambda_{1}t_{0}}$, and $\alpha$ in \eqref{alpha1}.

\subsection{Escape the interval $(-L_{0},L_{1})$}\label{section interval}

Define the first time that the process \eqref{1dvjp} escapes the interval $(-L_{0},L_{1})$,
\begin{align}\label{tauinterval}
\tau
:=\inf\{t>0:X(t)\notin(-L_{0},L_{1})\}.
\end{align}
Notice that the earliest time that the process could reach $L_{1}$ is $L_{1}/v_{1}$ and the earliest time that the process could reach $-L_{0}$ is $L_{0}/v_{0}$. Hence, if $L_{1}/v_{1}<L_{0}/v_{0}$, then $\ts=L_{1}/v_{1}$, and 
\begin{align*}
\P(\ts<\tau<\ts(1+\eps),X(\tau)=-L)
=0\quad\text{if }\ts(1+\eps)<L_{0}/v_{0}.
\end{align*}
Therefore, we can merely apply the analysis from the previous subsection to conclude that \eqref{short} holds with $p=1$, $t_{0}=L_{1}/v_{1}$, $q=p_{1}e^{-\lambda_{1}t_{0}}$, and $\alpha$ in \eqref{alpha1} with $L$ replaced by $L_{1}$. The case where $L_{1}/v_{1}>L_{0}/v_{0}$ is similar.

Hence, consider the case that $L_{0}=L_{1}=L$, $v_{0}=v_{1}=v$. Therefore,
\begin{align}\label{q1dball}
\ts=L/v,
\quad
q=p_{0}e^{-\lambda_{0}\ts}+p_{1}e^{-\lambda_{1}\ts}
\end{align}
Following the argument of the previous subsection, it follows from symmetry that
\begin{align*}
\P(\ts<\tau<\ts(1+\eps))
&=(p_{0}+p_{1}\lambda_{1}\ts)e^{-\lambda_{1}\ts}\big[1-e^{-\lambda_{0}(\frac{\eps L}{2v})}\big]\\
&\quad+(p_{1}+p_{0}\lambda_{0}\ts)e^{-\lambda_{0}\ts}\big[1-e^{-\lambda_{1}(\frac{\eps L}{2v})}\big]+\O(\eps^{2})\\
&=(1-q)\alpha\eps+\O(\eps^{2}),
\end{align*}
where
\begin{align}\label{alpha1dball}
\alpha
=\frac{L \Big[\lambda_{1} e^{-\frac{\lambda_{0} L}{v}} (\lambda_{0} L p_{0}+p_{1} v)+\lambda_{0} e^{-\frac{\lambda_{1} L}{v}} (\lambda_{1} L p_{1}+p_{0} v)\Big]}{2 v^2 \Big[1-p_{0} e^{-\frac{\lambda_{0} L}{v}}-p_{1} e^{-\frac{\lambda_{1} L}{v}}\Big]}.
\end{align}
Summarizing, \eqref{short} holds with $p=1$, $t_{0}$ and $q$ in \eqref{q1dball}, and $\alpha$ in \eqref{alpha1dball}.

\subsection{Numerical simulation: mean and full distribution}\label{num1}

\begin{figure}[t]
\centering
\includegraphics[width=.9\linewidth]{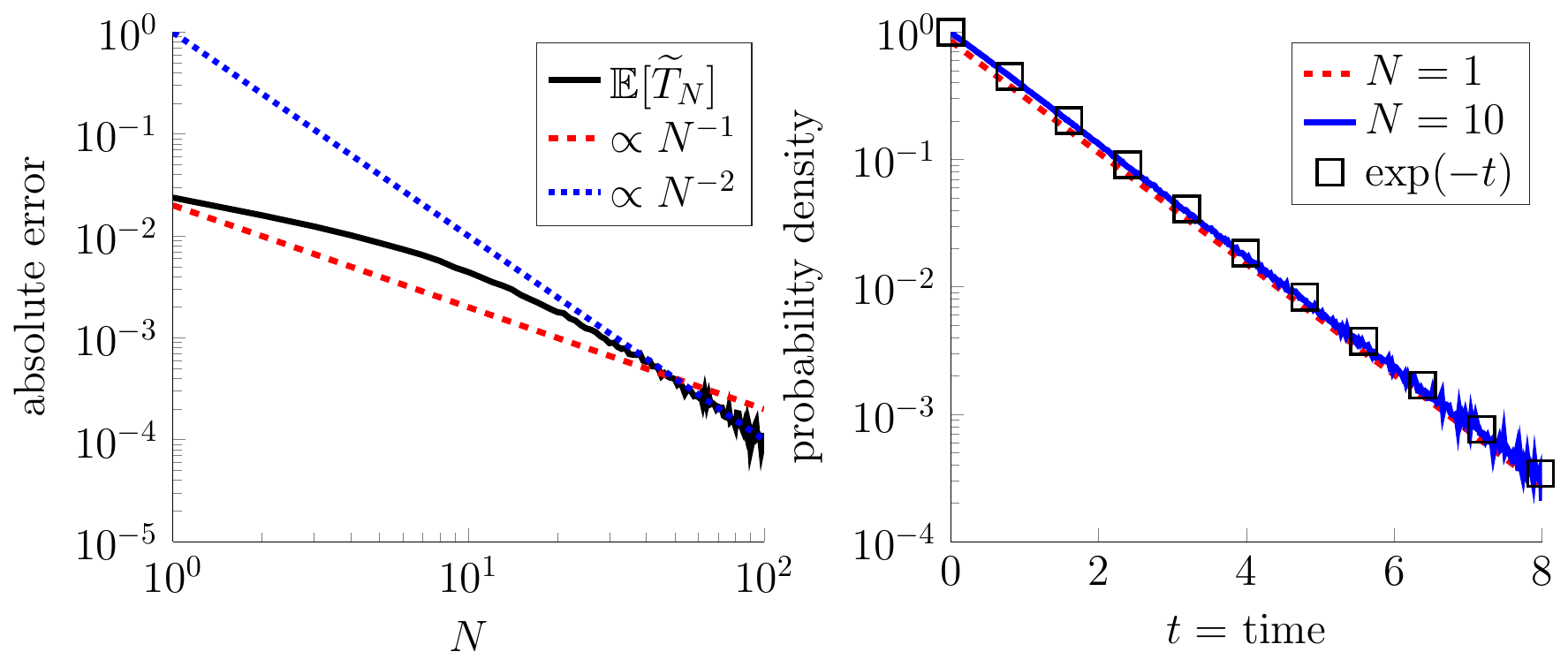}
\caption{Run and tumble in 1d. The left panel plots the absolute error \eqref{error1} for the mean of $\Tt_{N}$ as a function of $N$. The right panel plots the probability density of \eqref{density1} for $N=1,10$. In both panels, we take $L=v=1$ and $\lambda=3$.}
\label{fig1d}
\end{figure}

To illustrate our results, we perform stochastic simulations of the process \eqref{1dvjp} to generate statistically exact realizations of the fastest FPT, $T_{N}:=\{\tau_{1},\dots,\tau_{N}\}$, where $\{\tau_{1},\dots,\tau_{N}\}$ are $N$ iid realizations of the FPT in \eqref{tauinterval}. The details of our stochastic simulation algorithm are in the Appendix.

As in section~\ref{math}, recall that
\begin{align*}
T_{N}
=_{\dist}\ts+\xi_{N}(\Tt_{N}-\ts),
\end{align*}
where $\xi_{N}\in\{0,1\}$ satisfies $\P(\xi_{N}=1)=(1-q)^{N}$,
\begin{align*}
\Tt_{N}
:=\min\{\taut_{1},\dots,\taut_{N}\},
\end{align*}
and $\{\taut_{1},\dots,\taut_{N}\}$ are $N$ iid realizations of $\tau$ conditional that $\tau>\ts$. Since the distribution of $T_{N}$ is trivial in the case $\xi_{N}=0$, we focus on $\Tt_{N}$.

Consider the symmetric case that $L_{0}=L_{1}=L$, $v_{0}=v_{1}=v$, and $\lambda_{0}=\lambda_{1}=\lambda$. Since the interval $(-L,L)\subset\R$ is bounded, it is follows that $\E[T_{N}]<\infty$ for any $N\ge1$. To see this, note that the process must escape the interval, $|X(t)|>L$, if the jump process ever spends more than time $2L/v=2\ts$ in a particular state before jumping. Note further that each time the jump process jumps, the probability that it spends more than time $2\ts$ in its new state is $\overline{p}:=e^{-2\lambda\ts}\in(0,1)$. Hence, by conditioning on how many jumps $J$ takes before it spends more than time $2\ts$ in a particular state, we obtain
\begin{align}\label{finite}
\E[T_{N}]
\le\E[\tau_{1}]
\le2\ts\sum_{j=0}^{\infty}(j+1)\overline{p}(1-\overline{p})^{j}
=\frac{2\ts}{\overline{p}}<\infty.
\end{align}

Hence, Theorem~\ref{moments} implies that
\begin{align*}
\E[\Tt_{N}]
=\ts\Big(1+\frac{1}{\alpha N}\Big)
+o(\ts N^{-1})\quad\text{as }N\to\infty,
\end{align*}
where $\ts=L/v$ and $\alpha$ is in \eqref{alpha1dball}. In the left panel of Figure~\ref{fig1d}, we plot the absolute error of our approximation $\ts(1+\frac{1}{\alpha N})$,
\begin{align}\label{error1}
\Big|\E[\Tt_{N}]-\ts\Big(1+\frac{1}{\alpha N}\Big)\Big|,
\end{align}
as a function of $N$, where $\E[\Tt_{N}]$ is calculated from Monte Carlo simulations. In agreement with our theory, the absolute error decays faster than $N^{-1}$ as $N$ grows.

The right panel of Figure~\ref{fig1d} plots the probability density of
\begin{align}\label{density1}
\Sigma_{N}
=\frac{\Tt_{N}-\ts}{\ts a_{N}}
=\frac{\alpha N(\Tt_{N}-\ts)}{\ts}
\end{align}
calculated from Monte Carlo simulations for $N=1$ and $N=10$. In agreement with Theorem~\ref{main}, this distribution converges to an exponential distribution with unit mean as $N$ grows. Interestingly, this figure shows that the convergence to an exponential distribution is quite fast in this example.

\section{Two-dimensional run and tumble}\label{section 2d}

Consider a 2d run and tumble process $X(t)\in\R^{2}$ with constant speed $v>0$ and constant tumbling (switching) rate $\lambda>0$. That is, the process starts at $X(0)=0\in\R^{2}$ and moves at speed $v$ in a uniform random direction until an exponentially distributed time with rate $\lambda$, at which point the process chooses a new random direction and then moves at speed $v$ until another exponentially distributed time with rate $\lambda$, and so on. For simplicity, we assume that each random direction is chosen independently and uniformly, but it would be straightforward to relax this assumption and allow for correlations between successive random directions (see below).

Let $\tau$ be the first time the process escapes a disk of radius $L>0$,
\begin{align}\label{tau2d}
\tau
:=\inf\{t>0:\|X(t)\|>L\},
\end{align}
where $\|\cdot\|$ denotes the standard Euclidean length. By defining dimensionless time and space variables, $t\to (v/L)t$ and $x\to x/L$, we can without loss of generality consider a unit velocity and unit disk, $v=L=1$, and a general dimensionless switching rate $\lambda$. Hence, $\ts=1$.

As above, we have that
\begin{align}\label{sum2d}
\P(1<\tau<1+\eps)
=\sum_{j=0}^{\infty}\P(1<\tau<1+\eps,M(1)=j),
\end{align}
where $M(t)$ is the number of switches before time $t$. It is immediate that the $j=0$ term vanishes since $\tau=1$ if $M(1)=0$.

Let us now focus on the $j=1$ term. Notice that if $1<\tau<1+\eps$ and $M(1)=1$, then this puts a constraint on the angle that is chosen upon the first switch. To find this constraint, suppose the first switch occurs at time $s\in(0,1)$. Without loss of generality, we choose a coordinate system so that the initial direction of the process is directly in the vertical direction. That is, the initial direction of the process is in the $y$ direction,
\begin{align*}
X(t)=(0,t)\in\R^{2},\quad t\in[0,s].
\end{align*}
Notice that if $1<\tau<1+\eps$ and $M(1+\eps)=1$ and the first switch happens at time $s\in(0,1)$, then the location of the process at time $1+\eps$ will be on the disk of radius $1+\eps-s$ that is centered at $(0,s)\in\R^{2}$. That is,
\begin{align*}
X(1+\eps)\in
B:=\{(x,y)\in\R^{2}:x^{2}+(y-z)^{2}=(1+\eps-s)^{2}\}\subset\R^{2}.
\end{align*}
To find the probability that the angle chosen at time $s$ is such that $1<\tau<1+\eps$, we must find the fraction of the boundary of $B$ that lies outside the unit disk centered at the origin. An elementary geometry exercise yields that this fraction of the boundary is 
\begin{align}\label{frac2}
\begin{cases}
\pi^{-1}\arccos((x-s)/r) & \text{if }s\in(\eps/2,1),\\
1& \text{if }s\in(0,\eps/2],
\end{cases}
\end{align}
where $r=1+\eps-s$ and $x=(s^{2}-r^{2}+1)/(2s)$.

For simplicity, we have assumed that when the process switches direction, it chooses a new angle uniformly. Hence, the probability that $1<\tau<1+\eps$ given the first switch is at time $s\in(0,1)$ and that $M(1+\eps)=1$ is the fraction in \eqref{frac2}. Therefore, by conditioning on the time of the first switch, we obtain
\begin{align}
\P(1<\tau<1+\eps,M(1+\eps)=1)
&=\lambda e^{-\lambda}\Big[\frac{\eps}{2}+\frac{1}{\pi}\int_{\eps/2}^{1}\arccos((x-s)/r)\,\dd s\Big]\label{integral}\\
&=\lambda e^{-\lambda}\sqrt{2}\sqrt{\eps}+\O(\eps).\nonumber
\end{align}
We note that if we assumed that when the process switches direction, it chooses a new angle according to some nonuniform distribution (thus allowing for correlations between successive angles), this would merely change the integral in \eqref{integral}.

It is straightforward to check that the $j\ge2$ terms in \eqref{sum2d} are higher order, and so we obtain that \eqref{short} holds with $p=1/2$, $\ts=1$, $q=e^{-\lambda}$, and 
\begin{align}\label{alpha2d}
\alpha=\lambda e^{-\lambda}\sqrt{2}/(1-q).
\end{align}

\subsection{Numerical simulation: mean and full distribution}

\begin{figure}[t]
\centering
\includegraphics[width=.9\linewidth]{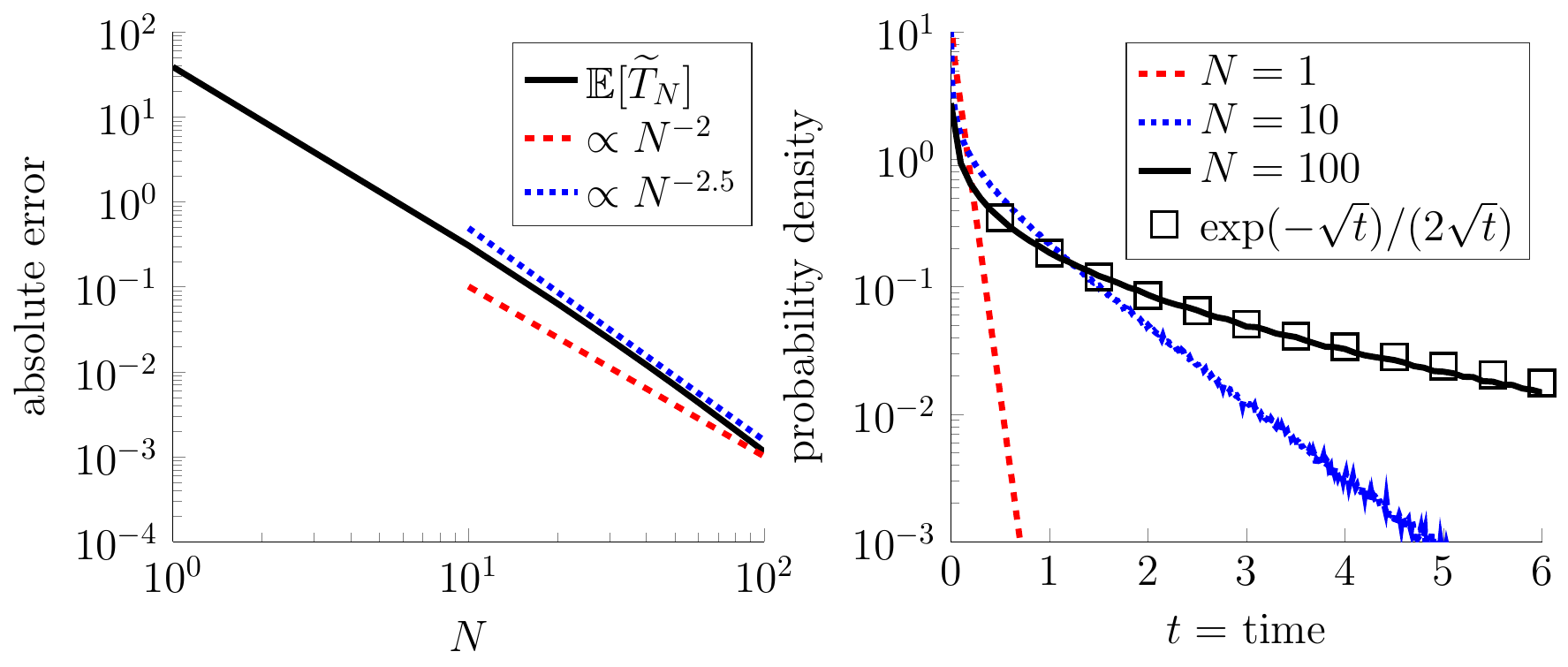}
\caption{Run and tumble in 2d. The left panel plots the absolute error \eqref{error2} for the mean of $\Tt_{N}$ as a function of $N$. The right panel plots the probability density of \eqref{density2} for $N=1,10,100$. In both panels, we take $\lambda=3$.}
\label{fig2d}
\end{figure}

As in section~\ref{num1}, we perform stochastic simulations of the 2d run and tumble process to generate statistically exact realizations of the fastest FPT, $T_{N}:=\{\tau_{1},\dots,\tau_{N}\}$, where $\{\tau_{1},\dots,\tau_{N}\}$ are $N$ iid realizations of the FPT in \eqref{tau2d}. We again focus on $\Tt_{N}$.

It follows from the same argument as in \eqref{finite} that $\E[T_{N}]<\infty$. Therefore, since $p=1/2$ in this case, Theorem~\ref{moments} implies that
\begin{align*}
\E[\Tt_{N}]
=\ts\Big(1+\frac{2}{(\alpha N)^{2}}\Big)
+o(\ts N^{-2})\quad\text{as }N\to\infty,
\end{align*}
where $\ts=1$ and $\alpha$ is in \eqref{alpha2d}. In the left panel of Figure~\ref{fig1d}, we plot the absolute error of our approximation $\ts(1+\frac{2}{(\alpha N)^{2}})$,
\begin{align}\label{error2}
\Big|\E[\Tt_{N}]-\ts\Big(1+\frac{2}{(\alpha N)^{2}}\Big)\Big|,
\end{align}
as a function of $N$, where $\E[\Tt_{N}]$ is calculated from Monte Carlo simulations. In agreement with our theory, the absolute error decays faster than $N^{-2}$ as $N$ grows.

The right panel of Figure~\ref{fig2d} plots the probability density of
\begin{align}\label{density2}
\Sigma_{N}
=\frac{\Tt_{N}-\ts}{\ts a_{N}}
=\frac{(\alpha N)^{2}(\Tt_{N}-\ts)}{\ts}
\end{align}
calculated from Monte Carlo simulations for $N=1,10,100$. In agreement with Theorem~\ref{main}, this density converges to $\exp(-\sqrt{t})/(2\sqrt{t})$ as $N$ grows.

\section{Three-dimensional run and tumble}\label{section 3d}

Now consider a 3d run and tumble process $X(t)\in\R^{3}$ with constant speed $v>0$ and constant tumbling (switching) rate $\lambda>0$. That is, the process starts at $X(0)=0\in\R^{3}$ and moves at speed $v$ in a uniform random direction until an exponentially distributed time with rate $\lambda$, at which point the process chooses a new random direction and then moves at speed $v$ until another exponentially distributed time with rate $\lambda$, and so on. As in section~\ref{section 2d}, we assume that each random direction is chosen independently and uniformly, but it would be straightforward to relax this assumption and allow for correlations between successive random directions (see below).

Let $\tau$ be the first time the process escapes a sphere of radius $L>0$,
\begin{align}\label{tau3d}
\tau
:=\inf\{t>0:\|X(t)\|>L\},
\end{align}
where $\|\cdot\|$ denotes the standard Euclidean length. Again, by defining dimensionless time and space variables, $t\to (v/L)t$ and $x\to x/L$, we can without loss of generality consider a unit velocity and unit sphere, $v=L=1$, and a general dimensionless switching rate $\lambda$. Hence, $\ts=1$.

As above, we have that
\begin{align}\label{sum3}
\P(1<\tau<1+\eps)
=\sum_{j=0}^{\infty}\P(1<\tau<1+\eps,M(1)=j).
\end{align}
It is immediate that the $j=0$ term vanishes since $\tau=1$ if $M(1)=0$. Hence, let us focus on the $j=1$ term in \eqref{sum3}. Notice that if $1<\tau<1+\eps$ and $M(1)=1$, then this puts a constraint on the angle that is chosen upon the first switch. To find this constraint, suppose the first switch occurs at time $s\in(0,1)$. Without loss of generality, we choose a coordinate system so that the initial direction of the process is directly toward the north pole. That is, the initial direction of the process is in the $z$ direction,
\begin{align*}
X(t)=(0,0,t)\in\R^{3},\quad t\in[0,s].
\end{align*}
Notice that if $1<\tau<1+\eps$ and $M(1+\eps)=1$ and the first switch happens at time $s\in(0,1)$, then the location of the process at time $1+\eps$ will be on the sphere of radius $1+\eps-s$ that is centered at $(0,0,s)\in\R^{3}$. That is,
\begin{align*}
X(1+\eps)\in
B:=\{(x,y,z)\in\R^{3}:x^{2}+y^{2}+(z-s)^{2}=(1+\eps-s)^{2}\}\subset\R^{3}.
\end{align*}
To find the probability that the angle chosen at time $s$ is such that $1<\tau<1+\eps$, we must find the fraction of the surface area of $B$ that lies outside the unit sphere centered at the origin. An elementary geometry exercise yields that this fraction of surface area is
\begin{align}\label{frac3}
\begin{cases}
\frac{\eps(2+\eps)}{4s(1-s+\eps)} & \text{if }s\in(\eps/2,1),\\
1& \text{if }s\in(0,\eps/2].
\end{cases}
\end{align}

As in the 2d run and tumble above, we have assumed that when the process switches direction, it chooses a new angle uniformly. Hence, the probability that $1<\tau<1+\eps$ given the first switch is at time $s\in(0,1)$ and that $M(1+\eps)=1$ is the fraction in \eqref{frac3}. Therefore, 
\begin{align}
&\P(1<\tau<1+\eps,M(1+\eps)=1)\nonumber\\
&\quad=\lambda e^{-\lambda}\Big[\frac{\eps}{2}+\int_{\eps/2}^{1}\frac{\eps(2+\eps)}{4s(1-s+\eps)}\,\dd s\Big]\label{integral3}\\
&\quad=\frac{1}{2} \lambda e^{-\lambda }  \eps  (-2 \ln (\eps )+1+\ln (2))+\O(\eps^{2}\ln(1/\eps)).\nonumber
\end{align}
If we assumed that when the process switches direction, it chooses a new angle according to some nonuniform distribution (thus allowing for correlations between successive angles), this would merely change the integral in \eqref{integral3}.

It is straightforward to check that the $j\ge2$ terms in \eqref{sum3} are higher order, and so 
\begin{align*}
\P(1<\tau<1+\eps)
&=\lambda e^{-\lambda}\ln(1/\eps)\eps+\O(\eps).
\end{align*}
Hence, we obtain that \eqref{shortlog} holds with $p=1$, $\ts=1$, $q=e^{-\lambda}$, and 
\begin{align*}
\alpha=\lambda e^{-\lambda}/(1-q).
\end{align*}

\subsection{Numerical simulation: mean and full distribution}

\begin{figure}[t]
\centering
\includegraphics[width=.9\linewidth]{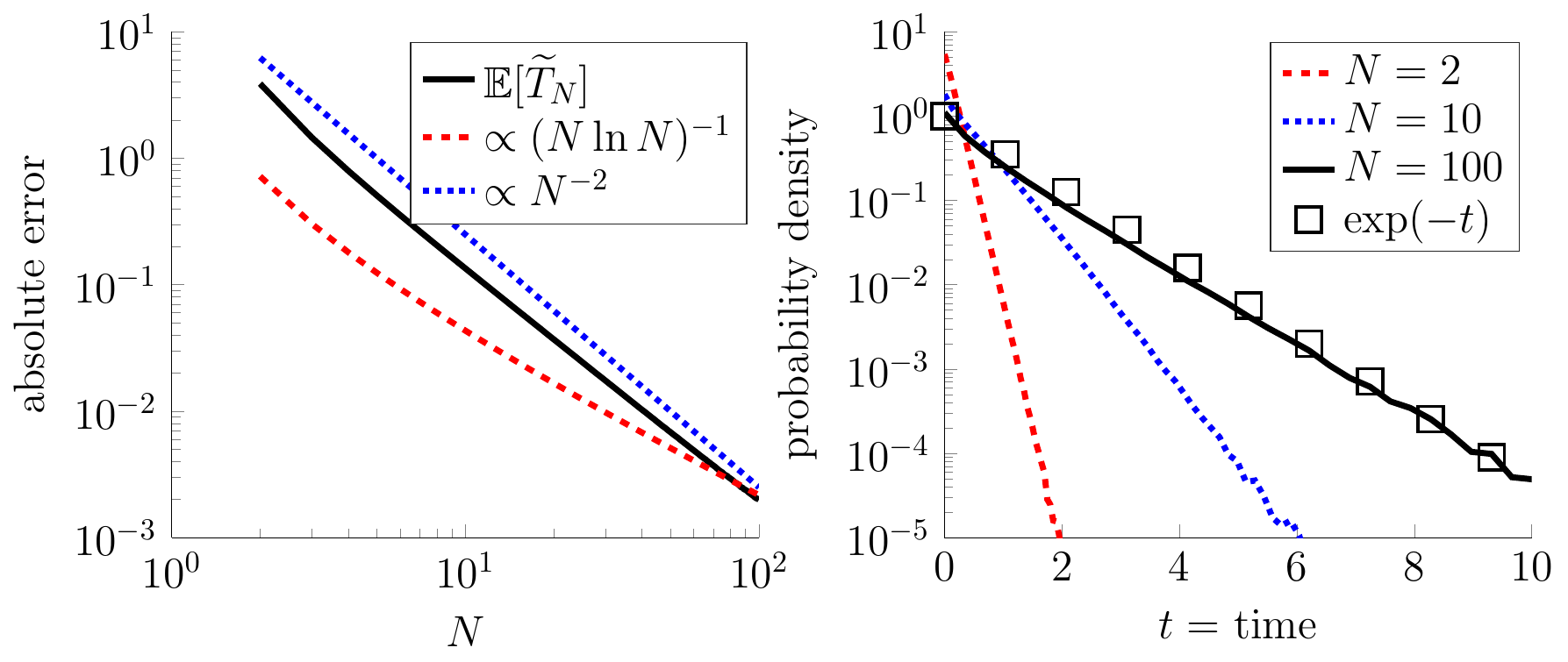}
\caption{Run and tumble in 3d. The left panel plots the absolute error \eqref{error3} for the mean of $\Tt_{N}$ as a function of $N$. The right panel plots the probability density of \eqref{density3} for $N=2,10,100$. In both panels, we take $\lambda=3$.}
\label{fig3d}
\end{figure}

As in section~\ref{num1}, we perform stochastic simulations of the 3d run and tumble process to generate statistically exact realizations of the fastest FPT, $T_{N}:=\min\{\tau_{1},\dots,\tau_{N}\}$ in \eqref{TN}, where $\{\tau_{1},\dots,\tau_{N}\}$ are $N$ iid realizations of the FPT in \eqref{tau3d}. We again focus on $\Tt_{N}$.

It follows from the same argument as in \eqref{finite} that $\E[T_{N}]<\infty$. Therefore, Theorem~\ref{moments} implies that
\begin{align*}
\E[\Tt_{N}]
=\ts\Big(1+\frac{1}{\alpha N\ln N}\Big)
+o(\ts (N\ln N)^{-1})\quad\text{as }N\to\infty,
\end{align*}
where $\ts=1$ and $\alpha$ is in \eqref{alpha2d}. In the left panel of Figure~\ref{fig1d}, we plot the absolute error of our approximation $\ts(1+\frac{1}{\alpha N\ln N})$,
\begin{align}\label{error3}
\Big|\E[\Tt_{N}]-\ts\Big(1+\frac{1}{\alpha N\ln N}\Big)\Big|,
\end{align}
as a function of $N$, where $\E[\Tt_{N}]$ is calculated from Monte Carlo simulations. In agreement with our theory, the absolute error decays faster than $1/(N\ln N)$ as $N$ grows.

The right panel of Figure~\ref{fig3d} plots the probability density of
\begin{align}\label{density3}
\Sigma_{N}
=\frac{\Tt_{N}-\ts}{\ts a_{N}}
=\frac{\alpha N\ln N(\Tt_{N}-\ts)}{\ts}
\end{align}
calculated from Monte Carlo simulations for $N=2,10,100$. In agreement with Theorem~\ref{main}, this distribution converges to an exponential distribution with unit mean as $N$ grows.

\section{One-dimensional linear PDMP}\label{section 4B}

In each of the previous examples, the velocity was constant between jumps. We now consider an example in which the velocity varies continuously between jumps. Specifically, consider the linear PDMP
\begin{align*}
\frac{\dd}{\dd t}X(t)
&=\begin{cases}
\mu(A_{0}-X(t)) & \text{if }J(t)=0,\\
\mu(A_{1}-X(t)) & \text{if }J(t)=1,
\end{cases}
\end{align*}
where $J(t)\in\{0,1\}$ is a continuous-time Markov jump process with jump rate $\lambda>0$ and $A_{0}\neq A_{1}$. This PDMP has been used to study both gene expression \cite{smiley2010} and storage systems \cite{boxma2005}.

Since we can nondimensionalize the problem by shifting and rescaling space $x\to(x-A_{0})/(A_{1}-A_{0})$ and rescaling time $t\to\mu t$, it is enough to consider the dimensionless problem
\begin{align}\label{Boxma}
\frac{\dd}{\dd t}X(t)
&=\begin{cases}
-X(t) & \text{if }J(t)=0,\\
1-X(t) & \text{if }J(t)=1,
\end{cases}
\end{align}
where $J(t)\in\{0,1\}$ now jumps with a general dimensionless rate $\lambda>0$. For simplicity, suppose $X(0)=1$ (the case of a general initial condition is similar). 

Define the first time $X(t)$ reaches some threshold $\theta\in(0,1)$,
\begin{align}\label{tau4B}
\tau
:=\inf\{t>0:X(t)=\theta\}.
\end{align}
Hence, the fastest $X(t)$ could reach $\theta$ is
\begin{align*}
\ts
=\ln(1/\theta)>0,
\end{align*}
and
\begin{align*}
q:=\P(\tau=\ts)
=p_{0}e^{-\lambda\ts}\in[0,1),
\end{align*}
where $p_{0}:=\P(J(0)=0)=:1-p_{1}$. In words, $q$ is the probability that the process starts in the ``down'' direction ($p_{0}$) and that it does not switch before it reaches the threshold ($e^{-\lambda t_{0}}$). Figure~\ref{figschem4} illustrates many realizations of this process.

\begin{figure}[t]
\centering
\includegraphics[width=.5\linewidth]{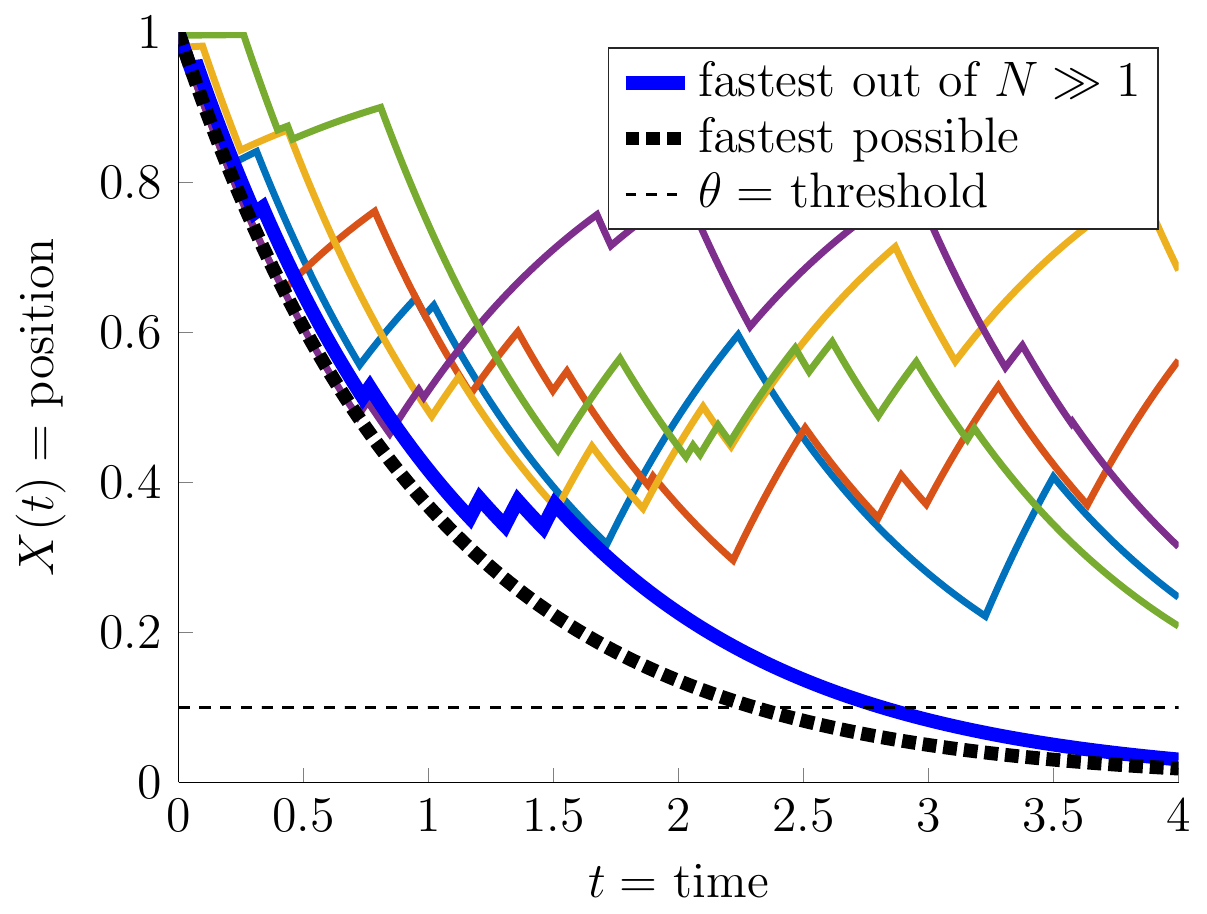}
\caption{Realizations of the linear PDMP in \eqref{Boxma}. The thin curves illustrate typical realizations of \eqref{Boxma}, the thick dashed black curve illustrates the fastest a realization could reach the threshold $\theta\in(0,1)$, and the thick blue curve illustrates the fastest realization to reach $\theta$ out of $N\gg1$ trajectories.}
\label{figschem4}
\end{figure}

As above, it is straightforward to check that
\begin{align*}
\P(\ts<\tau<\ts(1+\eps))
&=\sum_{j=0}^{\infty}\P(\ts<\tau<\ts(1+\eps),M(\ts)=j)\\
&=\P(\ts<\tau<\ts(1+\eps),M(\ts)=2,J(0)=0)\\
&\quad+\P(\ts<\tau<\ts(1+\eps),M(\ts)=1,J(0)=1)+\O(\eps)^{2}.
\end{align*}
Now, if $\ts<\tau<\ts(1+\eps)$, $J(0)=0$, and the first jump occurs at time $s_{1}>0$, then the next sojourn time $s_{2}$ in state $J=1$ must be such that
\begin{align*}
e^{-(\ts(1+\eps)-s_{1}-s_{2})}(1-(1-e^{-s_{1}})e^{-s_{2}})<\theta.
\end{align*}
Solving this for $s_{2}$, we find that
\begin{align*}
0<s_{2}<\ln \left(\left(e^{s_{1}}-1\right) \theta ^{\eps}+1\right)-s_{1}-\eps \ln (\theta ).
\end{align*}
By conditioning on the values of $s_{1}$ and $s_{2}$, we therefore obtain
\begin{align*}
&\P(\ts<\tau<\ts(1+\eps),M(\ts(1+\eps))=2,J(0)=0)\\
&\quad=p_{0}\int_{0}^{\ts}\int_{0}^{\ln \left(\left(e^{s_{1}}-1\right) \theta ^{\eps}+1\right)-s_{1}-\eps \ln (\theta )}\lambda e^{-\lambda s_{1}}\lambda e^{-\lambda s_{2}}e^{-\lambda(\ts-s_{1}-s_{2})}\,\dd s_{2}\,\dd s_{1}\\
&\quad=p_{0}\lambda^{2}\theta^{\lambda}(\theta-1)\ln(\theta)\eps+\O(\eps^{2}).
\end{align*}

Similary, if $\ts<\tau<\ts(1+\eps)$, $J(0)=1$, $M(\ts(1+\eps))=1$ and the first two sojourn times are $s_{1}$ and $s_{2}$, then
\begin{align*}
e^{-(\ts(1+\eps)-s_{1})}<\theta
\quad\text{and}\quad
s_{2}>\ts(1+\eps)-s_{1}.
\end{align*}
Hence,
\begin{align*}
\P(\ts<\tau<\ts(1+\eps),M(\ts)=1,J(0)=1)
&=(1-e^{-\lambda\eps\ln(1/\theta)})e^{-\lambda\ln(1/\theta)}+\O(\eps^{2})\\
&=\lambda\theta^{\lambda}\ln(1/\theta)\eps+\O(\eps^{2}).
\end{align*}

Therefore, we obtain that \eqref{short} holds with $p=1$, $\ts=\ln(1/\theta)$, $q=p_{0}e^{-\lambda \ts}$, and 
\begin{align}\label{alpha4B}
\alpha=\Big[p_{0}\lambda^{2}\theta^{\lambda}(1-\theta)\ln(1/\theta)
+p_{1}\lambda\theta^{\lambda}\ln(1/\theta)\Big]/(1-q).
\end{align}

\subsection{Numerical simulation: mean and full distribution}

\begin{figure}[t]
\centering
\includegraphics[width=.9\linewidth]{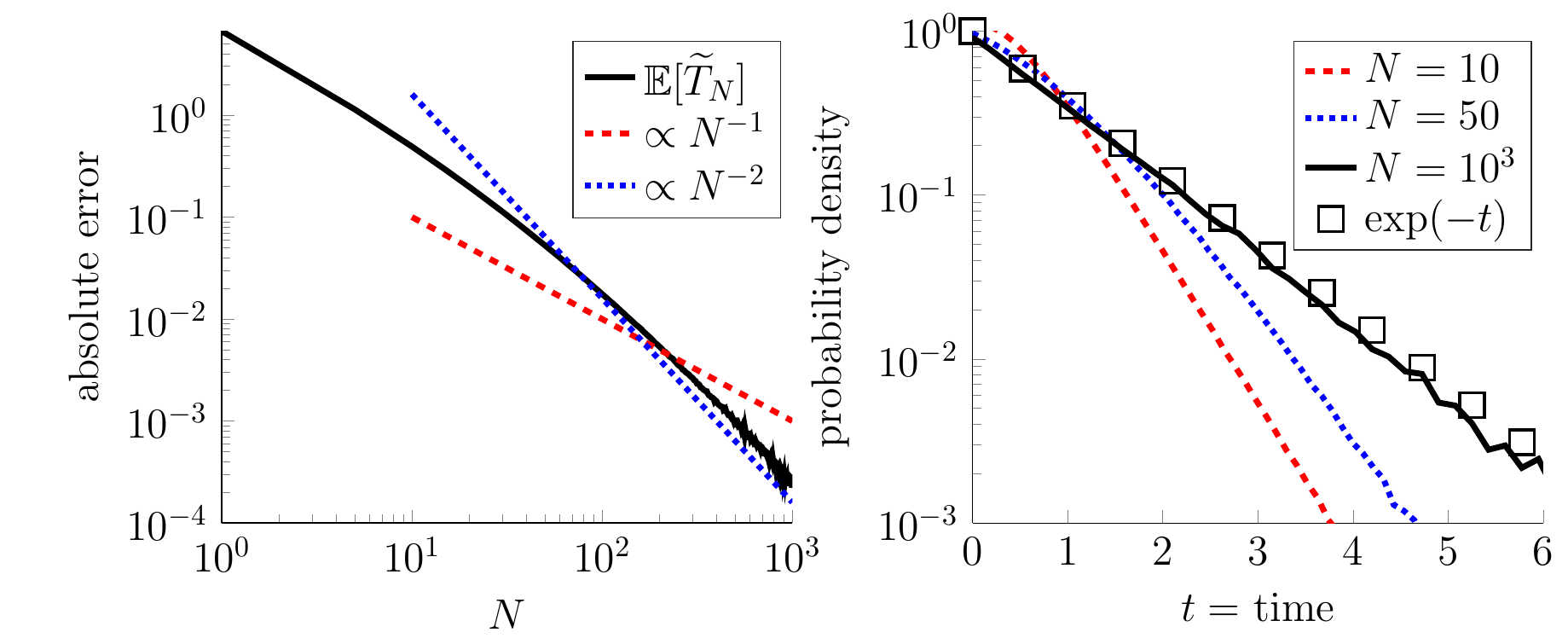}
\caption{Linear PDMP in \eqref{Boxma}. The left panel plots the absolute error \eqref{error4} for the mean of $\Tt_{N}$ as a function of $N$. The right panel plots the probability density of \eqref{density4} for $N=10,50,10^{3}$. In both panels, we take $\lambda=3$ and $p_{0}=p_{1}=0.5$.}
\label{fig4B}
\end{figure}

As in section~\ref{num1}, we perform stochastic simulations of the process \eqref{Boxma} to generate statistically exact realizations of the fastest FPT, $T_{N}:=\{\tau_{1},\dots,\tau_{N}\}$, where $\{\tau_{1},\dots,\tau_{N}\}$ are $N$ iid realizations of the FPT in \eqref{tau4B}. We again focus on $\Tt_{N}$.

It follows from a very similar argument to \eqref{finite} that $\E[T_{N}]<\infty$. Therefore, Theorem~\ref{moments} implies that
\begin{align*}
\E[\Tt_{N}]
=\ts\Big(1+\frac{1}{\alpha N}\Big)
+o(\ts N^{-1})\quad\text{as }N\to\infty,
\end{align*}
where $\ts=\ln(1/\theta)$ and $\alpha$ is in \eqref{alpha4B}. In the left panel of Figure~\ref{fig4B}, we plot the absolute error of our approximation $\ts(1+\frac{1}{\alpha N})$,
\begin{align}\label{error4}
\Big|\E[\Tt_{N}]-\ts\Big(1+\frac{1}{\alpha N}\Big)\Big|,
\end{align}
as a function of $N$, where $\E[\Tt_{N}]$ is calculated from Monte Carlo simulations. In agreement with our theory, the absolute error decays faster than $N^{-1}$ as $N$ grows.

The right panel of Figure~\ref{fig4B} plots the probability density of
\begin{align}\label{density4}
\Sigma_{N}
=\frac{\Tt_{N}-\ts}{\ts a_{N}}
=\frac{(\alpha N)^{2}(\Tt_{N}-\ts)}{\ts}
\end{align}
calculated from Monte Carlo simulations for $N=10,50,10^{3}$. In agreement with Theorem~\ref{main}, this distribution converges to an exponential distribution with unit mean as $N$ grows.

\section{Discussion}

In this paper, we studied extreme FPTs of PDMPs. We proved general theorems which yield the approximate distribution of such extreme FPTs based on the short time asymptotic behavior of a single FPT. We then applied these general results to some canonical PDMPs.

In addition to the numerous processes which have been modeled by PDMPs \cite{bressloff2017, bena06, teel15, yinbook}, this work was motivated by results on extreme FPTs of diffusive searchers. For the case of $N\gg1$ searchers diffusing in 1d with diffusivity $D>0$, Weiss et al.\ \cite{weiss1983} showed in 1983 that the mean fastest FPT satisfies
\begin{align}\label{result}
\E[T_{N}]
\sim\frac{L^{2}}{4D\ln N}\quad\text{as }N\to\infty,
\end{align}
where $L>0$ is the distance from the initial searcher locations to the target. The asymptotic behavior \eqref{result} was extended to various effectively 1d domains in \cite{yuste1996, yuste2000, yuste2001, van2003, redner2014, meerson2015} and extended to certain higher dimensional domains in \cite{ro2017, basnayake2019, lawley2019esp}. In \cite{lawley2019uni}, classical results in large deviation theory were used to show that \eqref{result} holds in great generality (including diffusions in $\R^{d}$ with space-dependent diffusivities and force fields and diffusions on $d$-dimensional Riemannian manifolds containing reflecting obstacles). In \cite{lawley2019gumbel}, extreme value theory was applied to the case of diffusive searchers to find the asymptotic probability distribution of $T_{N}$ and to find higher order corrections to \eqref{result}. In particular, the short time behavior of the distribution of a single diffusive FPT was used to determine the distribution of $T_{N}$ for large $N$ \cite{lawley2019gumbel}. The basic approach of the present work follows that of \cite{lawley2019gumbel}, except that the short time distribution of a PDMP FPT differs markedly from that of a diffusive FPT.

As described in the Introduction section, the fact that the fastest FPT of diffusive searchers vanishes for large $N$ is at odds with the fact that the time it takes searchers moving at finite speed to find a target is bounded above zero. In particular, run and tumble searchers with speed $v>0$ and tumbling rate $\lambda>0$ cannot reach a target distance $L>0$ away in a time less than $\ts:=L/v>0$. In what parameter regime can the extreme FPT of run and tumble processes be approximated by an extreme FPT of diffusion? Combining \eqref{result} with the bound $T_{N}\ge\ts$ suggests that a valid approximation requires that $\ts\ll L^{2}/(4D\ln N)$, which means
\begin{align}\label{condition}
\frac{D\ln N}{vL}\ll 1.
\end{align}
Upon identifying $D$ with $v^{2}/\lambda$, the requirement \eqref{condition} becomes
\begin{align}\label{condition2}
\frac{v\ln N}{\lambda L}\ll 1.
\end{align}
Therefore, if \eqref{condition2} is violated, we expect that run and tumble extreme FPTs that are calculated using the diffusion approximation are not valid. Modeling the motion of bacteria by run and tumble, approximate parameter values are $v\approx 10\,\mu\text{m}\,\text{s}^{-1}$ and $\lambda\approx 1\,\text{s}^{-1}$ \cite{berg1972}, which implies that \eqref{condition2} is violated for lengthscales $L<10^{3}\,\mu\text{m}$ and $N>10^{2}$. Similarly, approximate parameter values for modeling molecular motor transport by run and tumble are $v\approx 10^{-1}\,\mu\text{m}\,\text{s}^{-1}$ and $\lambda\approx 1\,\text{s}^{-1}$ \cite{newby2010}, which implies that \eqref{condition2} is violated for lengthscales $L<10\,\mu\text{m}$ and $N>10^{2}$.

Finally, it is interesting to discuss our results on the 3d run and tumble process in the context of sperm cells searching for the oocyte in human fertilization (as described in the Introduction). The distance a sperm cell must travel to find the oocyte is estimated at $L=10^{5}\,\mu\text{m}$ and they swim at an estimated speed of $v=75\,\mu\text{m}\,\text{s}^{-1}$ \cite{yang2016}. In \emph{in vivo} movies, sperm cell ``motions were mostly ballistic for distances of millimeters, as long as they do not encounter any obstacle'' \cite{yang2016}. Supposing that this ballistic motion distance is $l\in[4,5]\,\text{mm}$, we estimate that $\lambda$ is in the range
\begin{align}\label{lambdarange}
\lambda=\frac{v}{l}\in[0.015,0.01875]\,\text{s}^{-1}.
\end{align}
In Figure~\ref{fig5app}, we plot the formula that we derived in section~\ref{section 3d},
\begin{align*}
\E[T_{N}]
&\approx\frac{L}{v}\left[1+\Big(\frac{(1-e^{-{(\lambda L/v)}})^{N+1}}{{(\lambda L/v)} e^{-{(\lambda L/v)}} }\Big)\frac{1}{N\ln N}\right],
\end{align*}
for $\lambda$ in the range \eqref{lambdarange}. The solid blue curve is for the typical number of sperm cells, $N=3\times10^{8}$, and the dotted red curve is for the reduction $N=\frac{1}{4}(3\times10^{8})$. The dashed black line is the fastest possible search time of $L/v\approx20$ minutes, which is in the physiological timescale of tens of minutes \cite{khurana2018}. Of course, fertilization is incredibly complex \cite{eisenbach2006} and the process in section~\ref{section 3d} is a woefully crude idealization. Nevertheless, it is interesting that for typical physiological parameters, the value of $\E[T_{N}]$ is in the correct timescale (solid blue curve in Figure~\ref{fig5app}). It is also interesting that merely reducing $N$ by a factor of four can significantly increase the search time (red dotted curve in Figure~\ref{fig5app}), given that such a reduction may cause infertility \cite{bensdorp2007}. More broadly, this simple calculation emphasizes the importance of the number of searchers in determining FPTs and that a very large number of searchers ($\sim10^{8}$) may be necessary in certain biophysical processes. 

\begin{figure}[t]
\centering
\includegraphics[width=.5\linewidth]{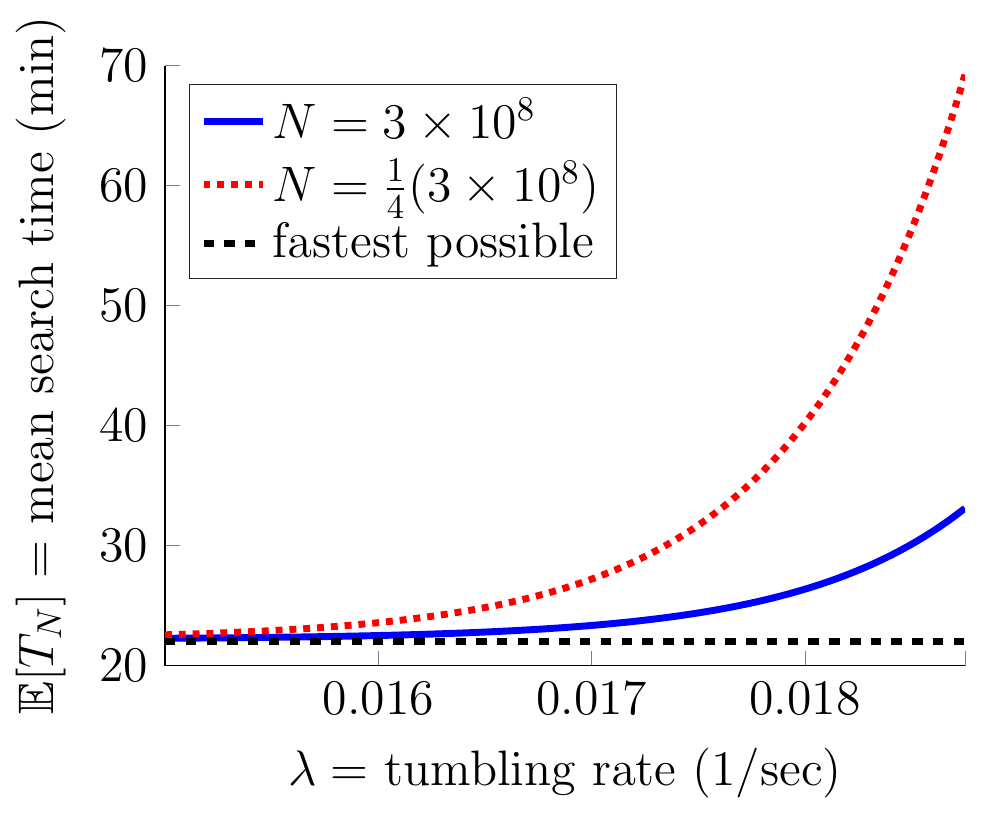}
\caption{Expected search time for 3d run and tumble with parameters estimated for human fertilization. See the text for details.}
\label{fig5app}
\end{figure}

\section{Appendix}

The appendix is divided into three sections. In section~\ref{proofs}, we give the proofs of the propositions and theorems of section~\ref{math}. In section~\ref{review}, we review the diffusion approximation for 1d run and tumble processes described in the Introduction. In section~\ref{algo}, we describe the stochastic simulation algorithm used in sections~\ref{section 1d}-\ref{section 4B}.

\subsection{Proofs}\label{proofs}

\begin{proof}[Proof of Proposition~\ref{basic}]
This proposition presents basic facts about Weibull random variables which are straightforward to show using \eqref{xweibull}.
\end{proof}

\begin{proof}[Proof of Theorem~\ref{main}]
In light of \eqref{tilde}, define the sequence of random variables $\{X_{n}\}_{n\ge1}$ by $X_{n}=(\ts-\taut_{n})/\ts$, where the distribution of $\taut_{n}$ is given in \eqref{ccdf}. Therefore, for $-1\ll x<0$ we have
\begin{align}\label{F}
\begin{split}
F(x)
&:=\P(X_{n}\le x)
=\P(\taut_{n}\ge\ts(1-x))\\
&=1-\alpha(\chi+(1-\chi)\ln(-1/x))(-x)^{p}\\
&\quad+o((\chi+(1-\chi)\ln(-1/x))(-x)^{p})\quad\text{as }x\to0-,
\end{split}
\end{align}
where $\chi=1$ if \eqref{short} holds and $\chi=0$ if \eqref{shortlog} holds.
Hence, if $y>0$, then \eqref{F} implies
\begin{align*}
\lim_{t\to0+}\frac{1-F(-ty)}{1-F(-t)}
=y^{p}.
\end{align*}
Therefore, Theorem 1.2.1 and Corollary 1.2.4 in \cite{haanbook} imply that
\begin{align*}
M_{N}:=\max\{X_{1},\dots,X_{N}\}
=-\min\{-X_{1},\dots,-X_{N}\}
=\frac{-(\Tt_{N}-t_{0})}{t_{0}}
\end{align*}
satisfies
\begin{align}\label{pcd}
\frac{-M_{N}}{a_{N}}
\to_{\dist}\textup{Weibull}(1,p)\quad\text{as }N\to\infty,
\end{align}
where $a_{N}>0$ satisfies 
\begin{align*}
\alpha(\chi+(1-\chi)\ln(1/a_{N}))(a_{N})^{p}=1/N\quad\text{for sufficiently large $N$}.
\end{align*}
Solving this equation for $a_{N}$ yields
\begin{align*}
a_{N}
:=\begin{cases}
(\alpha N)^{-1/p} & \text{if \eqref{short} holds},\\
\big(-\alpha N W_{-1}(-p/(N\alpha))/p\big)^{-1/p} & \text{if \eqref{shortlog} holds},
\end{cases}
\end{align*}
where $W_{-1}(\cdot)$ denotes the lower branch of the LambertW function \cite{corless1996}.

Now, it is straightforward to check that if \eqref{pcd} holds for a sequence $\{a_{N}\}$, then it also holds for any sequence $\{a_{N}'\}$ satisfying
\begin{align*}
\lim_{N\to\infty}a_{N}/a_{N}'=1.
\end{align*}
Therefore, in the case that \eqref{shortlog} holds, it follows from the asymptotic behavior of the LambertW function \cite{corless1996} that we may take
\begin{align*}
\alpha_{N}
=(\alpha N\ln(N)/p)^{-1/p}.
\end{align*}
Defining $\Sigma_{N}:=-M_{N}/a_{N}$ completes the proof.
\end{proof}

\begin{proof}[Proof of Theorem~\ref{moments}]
We have assumed that $\E[T_{N}]<\infty$ for some $N\ge1$. Thus, if $m\in(0,1)$, then $\E[(T_{N})^{m}]\le1+\E[T_{N}]<\infty$. If $m>1$, then it is straightforward to show that (see for example the proof of Proposition 2 in \cite{lawley2019uni})
\begin{align*}
\E[(T_{2^{m-1}N})^{m}]<\infty.
\end{align*}
Since Weibull random variables have finite moments, we complete the proof by applying Theorem 2.1 in \cite{pickands1968}. We note that rather than appealing to  \cite{pickands1968}, we could prove this theorem using the argument in the proof of Theorem~\ref{kth moment} below.
\end{proof}

\begin{proof}[Proof of Proposition~\ref{basicg}]
This proposition presents basic facts about generalized Gamma random variables which are straightforward to show using \eqref{gGamma}.
\end{proof}

\begin{proof}[Proof of Theorem~\ref{kth}]
Define
\begin{align*}
\xi_{j,N}
:=\begin{cases}
1 & \text{if }T_{j,N}=\ts\text{ and }T_{j+1,N}>\ts,\\
0 & \text{otherwise},
\end{cases}
\end{align*}
where $T_{0,N}:=\ts=:T_{N+1,N}$. Define
\begin{align*}
\Tt_{k,N}
:=\min\big\{\{\taut_{1},\dots,\taut_{N}\}\backslash\cup_{j=1}^{k-1}\{\Tt_{j,N}\}\big\},\quad k\in\{1,\dots,N\},
\end{align*}
where $\Tt_{1,N}:=\Tt_{N}$. Then
\begin{align*}
T_{k,N}
=_{\dist}
\sum_{j=k}^{N}\xi_{j,N}\ts
+\sum_{j=0}^{k-1}\xi_{j,N}\Tt_{k-j,N}
=\ts
+\sum_{j=0}^{k-1}\xi_{j,N}(\Tt_{k-j,N}-\ts)
\end{align*}
since $\sum_{j=0}^{N}\xi_{j,N}=1$ almost surely. Defining $\Sigma_{k-j,N}:=(\Tt_{k-j,N}-\ts)/(\ts a_{N})$, the convergence in distribution in \eqref{cd1} follows immediately from Theorem~\ref{main} above and Theorem 3.5 in \cite{colesbook}.
\end{proof}

\begin{proof}[Proof of Theorem~\ref{kth moment}]
If a sequence of uniformly integrable random variables converges in distribution, then the moments of the random variables also converge \cite{billingsley2013}. Therefore, we need only prove that
\begin{align}\label{suffc}
\sup_{N}\E[(\Sigma_{k,N})^{2}]
=\sup_{N}\E\Big[\Big(\frac{\Tt_{k,N}-\ts}{\ts a_{N}}\Big)^{2}\Big]<\infty,
\end{align}
since \eqref{suffc} implies that the sequence $\{\frac{T_{k,N}-b_{N}}{a_{N}}\}_{N}$ is uniformly integrable \cite{billingsley2013}.

Letting
\begin{align*}
S(x)
:=\P\Big(\frac{\taut_{n}-\ts}{\ts}>x\Big),
\end{align*}
it follows by definition of $\Tt_{k,N}$ that
\begin{align*}
\P\Big(\frac{\Tt_{k,N}-\ts}{\ts}>x\Big)
=\sum_{j=0}^{k-1}{N\choose j}(1-S(x))^{j}(S(x))^{N-j}.
\end{align*}
Therefore,
\begin{align*}
\E\Big[\Big(\frac{\Tt_{k,N}-\ts}{\ts}\Big)^{2}\Big]
&=\int_{0}^{\infty}\P\Big(\frac{\Tt_{k,N}-\ts}{\ts}>\sqrt{x}\Big)\,\dd x\\
&=\sum_{j=0}^{k-1}{N\choose j}\int_{0}^{\infty}(1-S(\sqrt{x}))^{j}(S(\sqrt{x}))^{N-j}\,\dd x.
\end{align*}
Using \eqref{short}-\eqref{shortlog} and standard integral estimates ensures that \eqref{suffc} holds, which completes the proof.
\end{proof}

\subsection{Run and tumble diffusion approximation}\label{review}

For the 1d run and tumble described in the Introduction section, the probability density for its position, $p(x,t)$, can be decomposed into the densities of right-moving and left-moving searchers,
\begin{align*}
p(x,t)=p_{+}(x,t)+p_{-}(x,t),
\end{align*}
where $p_{\pm}$ satisfy the advection-reaction equations,
\begin{align*}
\frac{\partial}{\partial t}p_{+}+v\frac{\partial}{\partial x}p_{+}(x,t)
&=-\lambda (p_{+}-p_{-})\\
\frac{\partial}{\partial t}p_{-}-v\frac{\partial}{\partial x}p_{-}(x,t)
&=\lambda (p_{+}-p_{-}).
\end{align*}
Upon an algebraic manipulation and cross-differentiation of these equations, one obtains
\begin{align}\label{pp}
\frac{\partial^{2}}{\partial t^{2}}p
+2\lambda \frac{\partial}{\partial t}p
=v^{2}\frac{\partial^{2}}{\partial x^{2}}p.
\end{align}
If $\xc>0$ and $\tc>0$ denote respectively some lengthscale and timescale of interest, then nondimensionalizing \eqref{pp} in terms of $\overline{x}:=x/\xc$ and $\overline{t}:=t/\tc$ and rearranging yields
\begin{align*}
\frac{\partial}{\partial \overline{t}}p
=\frac{v^{2}\tc}{2\lambda \xc^{2}}\frac{\partial^{2}}{\partial \overline{x}^{2}}p
-\frac{1}{2\lambda\tc}\frac{\partial^{2}}{\partial \overline{t}^{2}}p.
\end{align*}
Assuming $\tc\gg1/\lambda$, $\xc=\O(\sqrt{(v^{2}/\lambda) \tc})$, and neglecting $\O(1/(\lambda \tc))$ terms from the equation then yields \eqref{FPE} in dimensional variables $x$ and $t$.

\subsection{Stochastic simulation algorithm}\label{algo}

We now describe the statistically exact stochastic simulation algorithm used in sections~\ref{section 1d}-\ref{section 4B}. Since $\tau=\ts$ with probability $q$, we describe our stochastic simulation algorithm to generate samples of $\tau$ conditioned that $\tau>\ts$ (the case $\tau=\ts$ is trivial). As above, we denote this conditional FPT by $\taut$ and note that it is defined by
\begin{align*}
\P(\taut<t)
=\frac{\P(\ts<\tau<t)}{1-q}.
\end{align*}

First consider the run and tumble processes in sections~\ref{section 1d}-\ref{section 3d}. For the 1d case in section~\ref{section 1d}, consider the problem of escape from an interval described in section~\ref{section interval} with $L_{0}=L_{1}=v_{0}=v_{1}=1$. For these run and tumble processes in 1d, 2d, and 3d, note that the initial direction is unimportant by symmetry. Further, since we are simulating the conditioned FPT $\taut$, the first jump time $s_{1}>0$ must occur before time $\ts=1$. Hence, the distribution of $s_{1}$ is exponential conditioned that $s_{1}<\ts$. That is, 
\begin{align}\label{s1}
\P(s_{1}>t)
=\begin{cases}
\frac{e^{-\lambda t}-e^{-\lambda \ts}}{1-e^{-\lambda \ts}} & \text{if }t>\ts,\\
0 & \text{if }t\le\ts.
\end{cases}
\end{align}
To generate a realization of $s_{1}$, we set $s_{1}=-\ln(U(1-e^{-\lambda\ts})+e^{-\lambda\ts})/\lambda$ where $U$ is uniformly distributed on $[0,1]$. We then generate realizations of the switching times and random directions until the process leaves a unit ball, which we then record as a single realization of $\taut$.

The stochastic simulation algorithm for the linear PDMP in section~\ref{section 4B} is very similar, except we also need to generate the random initial condition of the jump process $J$. If $p_{0},p_{1}$ denote the unconditioned probabilities, $p_{j}:=\P(J(0)=j)$, then the conditioned probabilities are
\begin{align*}
\widetilde{p}_{1}
&:=\P(J(0)=1\,|\,\tau>\ts)
=\frac{\P(J(0)=1,\tau>\ts)}{\P(\tau>\ts)}
=\frac{p_{1}}{p_{1}+p_{0}(1-e^{-\lambda \ts})},\\
\widetilde{p}_{0}
&:=\P(J(0)=1\,|\,\tau>\ts)=1-\widetilde{p}_{1}.
\end{align*}
Thus, to generate a realization of $\taut$, we first generate the initial condition $J(0)=j$ with probability $\widetilde{p}_{j}$. Then, if $J(0)=1$, then we generate the first switching time $s_{1}$ as an exponential random variable with rate $\lambda$. If $J(0)=0$, then the process must switch before time $\ts$, and so we generate $s_{1}$ as in \eqref{s1} above. We then continue to generate switching times until the process crosses the threshold $\theta$, which we then record as a single realization of $\taut$.

For each PDMP, we generate $M=10^{8}$ realizations of $\taut$, which we denote by $\{\taut_{1},\dots,\taut_{M}\}$. Then, for a given value of $N\ge1$, we obtain $K=\floor*{M/N}$ iid realizations of $\Tt_{N}$ (where $\floor*{\cdot}$ denotes the floor operator) denoted by $\Tt_{N}^{(1)},\dots,T_{N}^{(K)}$ by defining
\begin{align*}
\Tt_{N}^{(k)}
:=\min\{\taut_{(k-1)N},\taut_{(k-1)N+1},\dots,\taut_{kN}\},\quad k\in\{1,\dots,K\}.
\end{align*}
From these realizations of $\Tt_{N}$, we then compute the statistics (mean and empirical probability densities) used in Figures~\ref{fig1d}, \ref{fig2d}, \ref{fig3d}, and \ref{fig4B}.

\bibliography{library.bib}
\bibliographystyle{unsrt}

\end{document}